\documentclass[11pt,a4paper,english]{article}

\usepackage[a4paper,margin=2.5cm,vmargin={1.5cm,2.5cm},includeheadfoot]{geometry}
\usepackage[scaled=0.92]{helvet}
\usepackage[latin 1]{inputenc}

\usepackage{amssymb,amscd, amsmath, babel, amstext, amsthm, latexsym, enumerate}
\usepackage{color}
\usepackage[font=scriptsize]{caption}

\newtheorem{theorem}{Theorem}
\newtheorem{lemma}[theorem]{Lemma}
\newtheorem{corollary}[theorem]{Corollary}

\newtheorem{example}[theorem]{Example}
\newtheorem{definition}[theorem]{Definition}

\numberwithin{theorem}{section}

\newtheorem{lem}[theorem]{Lemma}
\newtheorem{remark}[theorem]{Remark}
\newtheorem{defn}[theorem]{Definition}
\numberwithin{equation}{section}

\newcommand{\R}{\mathbb{R}}

\newcommand{\F}{\ensuremath{\mathcal{F}}}

\newcommand{\X}{\ensuremath{\mathbb{X}}}

\DeclareMathOperator*{\esssup}{ess\,sup}
\DeclareMathOperator*{\essinf}{ess\,inf}

\newcommand{\HH}{\mathbb{H}}

\newcommand{\MM}{\mathbb{M}}
\newcommand{\Mm}{\mathcal{M}}

\newcommand{\E}{\mathbb{E}}
\newcommand{\Prob}{\mathbb{P}}
\newcommand{\Q}{\mathbb{Q}}
\newcommand{\ind}{\mathbf{1}}
\newcommand{\Rr}{{\mathbb{R}_0}}
\newcommand{\RR}{{\mathbb{R}}}

\newcommand{\Tt}{[0,T]}
\newcommand{\Ff}{\mathcal{F}}
\newcommand{\Gg}{\mathcal{G}}
\newcommand{\Qq}{\mathcal{Q}}

\newcommand{\FF}{\mathbb{F}}
\newcommand{\GG}{\mathbb{G}}
\newcommand{\Bb}{\mathcal{B}}
\newcommand{\Ht}{\tilde{H}}

\newcommand{\II}{\mathcal{I}}

\newcommand{\LL}{\mathcal{L}}
\newcommand{\ins}{\,}

\DeclareMathOperator{\Hg}{\mathcal{H}^\Gg_2}

\renewcommand{\theenumi}{\roman{enumi})}
\renewcommand{\labelenumi}{\theenumi}

\title{Hedging under worst-case-scenario\\ in a market driven by time-changed L\'evy noises}

\author{Giulia Di Nunno\thanks{CMA and Department of Mathematics,
University of Oslo, P.O. Box 1053 Blindern, N-0316 Oslo, and Department of Business and Management Science, NHH, Helleveien 30, N-5045 Bergen.
Email: giulian@math.uio.no}  \and
Erik Hove Karlsen\thanks{Department of Mathematics, University of Oslo, P.O. Box 1053 Blindern, N-0316 Oslo. Email: erikhk@math.uio.no.}}
\date{Version: April 20, 2015}

\begin{document}
\maketitle

\begin{abstract}
In an incomplete market driven by time-changed L\'evy noises we consider the problem of hedging a financial position coupled with the underlying risk of model uncertainty. Then we study hedging under worst-case-scenario.
The proposed strategies are not necessarily self-financing and include the interplay of a cost process to achieve the perfect hedge at the end of the time horizon.
The hedging problem is tackled in the framework of stochastic differential games and it is treated via backward stochastic differential equations. Two different information flows are considered and the solutions compared.

\vspace{2mm}\noindent

\vspace{2mm}\noindent

\vspace{2mm}\noindent

\end{abstract}

\section{Introduction}
In a stylized incomplete market on the time horizon $[0,T]$ $(T>0)$, we consider the problem of hedging a contingent claim coupled with the underlying risk of an uncertain model description. This is referred to as model ambiguity in the literature, see e.g. \cite{ChenEpstein2014}.
We choose to consider a conservative evaluation of such exposure to model ambiguity by measuring the robustness of the strategy in terms of a worst-case-scenario risk measure. 
In fact, on the complete filtered probability space $(\Omega, \Ff, \Prob)$, $\MM = \{ \Mm_t, \, t\in [0,T]\} $, we fix the dynamic coherent risk measure
\begin{equation}
\rho_t (\xi) := \esssup_{\Q \in \Qq_\MM} \: \E_\Q \big[ - \xi \vert \Mm_t \big], \quad t\in [0,T],
\label{risk-measure}
\end{equation}
where $\Qq_\MM$ is the set of all scenarios considered, which are given by the probability measures $\Q$ equivalent to $\Prob$ on the future outcomes $(\Omega, \Mm_T)$.
The risk measure $\rho$ is well-defined and finite for all financial positions $\xi$ integrable with respect to all $\Q\in \Qq_\MM$.
This risk measure is naturally connected to the concept of worst-case-scenario, associating  the risk evaluation with the highest expected value on the downside of the position over all $\Q\in\Qq_\MM$. 
For this reason it is also a conservative evaluation of the risk exposure.

\bigskip
The financial market presents two investment possibilities: a saving account with price dynamics given by
\begin{equation}
dS^{(0)} _t = r_t S^{(0)}_t dt, \quad S^{(0)}_0 = 1,
\label{bond}
\end{equation}
which is used as num\'eraire, and a stock with price dynamics of the type
\begin{equation}
dS^{(1)} _t = \alpha_t S^{(1)}_t dt + \sigma_t S^{(1)}_t dB_t + \int_\Rr \gamma_t(z)S^{(1)}_{t^-} \tilde H(dt,dz) , \quad S^{(1)}_0 > 1.
\label{stock}
\end{equation}
The driving noises $B$ and $\tilde H$ are a doubly stochastic Gaussian and Poisson type measures, respectively related to a time-changed Brownian and Poisson process.
The parameters $r$, $\alpha$, $\sigma$, $\gamma$ are c\`agl\`ad adapted stochastic processes and fields.
To ensure the existence of a square integrable positive solution and to allow further analysis, we assume that
$\vert r_t \vert <C$ (for some $C > 0$) $\Prob \times dt$-a.e., $\gamma_t(z) > -1$ $\Prob\times \Lambda $-a.e., and
$$
\E \Big[ \int_0^T \Big\{   \vert \alpha_t \vert + \sigma _t ^2\lambda^B_t + \int_\Rr \vert \ln (1+\gamma_t(z) ) - \gamma_t(z) \vert \nu(dz) \lambda^H _t \Big\} dt \Big]
< \infty.
$$
The elements $\nu$, $(\lambda^B, \lambda^H)$, and $\Lambda$ are associated to the behaviour of the jumps in these dynamics and the intensity of the (stochastic) time distortion applied, see the next section for details.

\bigskip
Price dynamics of this type include various well-known stochastic volatility price models such as \cite{carr2003}, where time-change of a pure jump L\'evy process is used to take the erratic behaviour of volatility into account by mimicking the transition between a real-time clock to a transactions-time clock. See also e.g. \cite{Barndorff-Nielsen2002}, \cite{heston1993}, \cite{hull1987}, \cite{stein1991} where the stochastic volatility models  lead to dynamics driven by time-changed Brownian motions, which are the doubly stochastic Gaussian noises in this paper.
Also in the credit risk literature we can find examples of price dynamics of the type above. See e.g. \cite{Lando1998}, where doubly stochastic Poisson processes, also called the Cox processes, are largely used in the modelling of prices subject to default risk.

\bigskip
In the market above, we consider a financial claim $F \in L^2(\Omega,\Mm_T,\Prob)$ with payoff at $T>0$ and the $\MM$-predictable hedging strategies represented by the triple $(\pi, V^\pi, C^\pi)$, i.e. $\pi_t$ is the wealth invested in the stock at $t$ representing the market portfolio, $V^\pi_t$ the value of the strategy on the market, and $C^\pi_t$ is the cost process. We denote $V^\pi_0=v>0$. 
We assume that 
\begin{equation}
\label{value}
\begin{split}
d V^\pi_t =  \: & \frac{\pi_t}{S^{(1)}_{t^-}} dS^{(1)}_t + \frac{V^\pi_t - \pi_t }{S^{(0)}_t} dS^{(0)}_t \\ 
= \: & \Big( V^\pi_t r_t + \pi_t \big( \alpha_t - r_t \big) \Big) dt + \pi_t \sigma_t dB_t + \int_\Rr \pi_t \gamma_t(z) \tilde H(dt,dz) \Big) 
\end{split}
\end{equation}
and that $C^\pi_t = \rho_t \big( e^{-\int_t^T r_sds} (V^\pi_T - F) \big).$
We assume that \eqref{value} admits a unique strong solution for any admissible $\pi$ (see e.g. \cite{Jacod1979} for conditions). Moreover, we require that the solution is square integrable, and for this we assume 
\begin{equation}
\label{nec-value}
\begin{split}
\E & \Big[ \int_0^T \Big\{   \vert \alpha_t -r_t \vert \vert \pi_t \vert  + \vert \alpha_t\sigma _t\vert^2 \lambda^B_t + \int_\Rr \vert \pi_t \gamma_t(z) \vert^2 \nu(dz) \lambda^H _t \Big\} dt \Big]
< \infty.
\end{split}
\end{equation}
Moreover, the process $C^\pi$ is also assumed square integrable. 

The process $Y^\pi_t:= V^\pi_t + C^\pi_t$, $t\geq 0$, is called the (total) price of the strategy.
The hedging of $F$ is obtained for a strategy $(\hat\pi, V^{\hat\pi}, C^{\hat\pi})$ that yields $Y^{\hat\pi}_T = F$.
Note that $Y^\pi_0 = v + C^\pi_0$ is $\Mm_0$-measurable and, if the hedging strategy has $C^\pi \equiv 0$, then the market investments are enough to self-finance the hedge. It is only in a complete market that it is possible to hedge all claims with these self-financing strategies.

\bigskip
We observe that for any hedging strategy $(\pi, V^{\pi}, C^{\pi})$ with $Y^\pi_T = F$ we have
\begin{equation*} 
\label{OP0}
\rho_t \big( e^{-\int_t^T r_sds} (Y^{\pi}_T - F) \big) = 0, \quad t\in [0,T].
\end{equation*}
This means that the risk given by the spread between the discounted final strategy price and the actual claim is zero according to the risk measure given, and actually all coherent risk measures.  

\bigskip
Hereafter, we consider the problem of finding a hedging strategy  $(\hat\pi, V^{\hat\pi}, C^{\hat\pi})$ for $F$ such that 
$Y^{\hat\pi}_t = Y_t$, $t \in [0,T]$, where
\begin{align}
\label{Y}
Y_t :=
 &\essinf_{\pi \in \Pi_\MM} \, \rho_t \Big( e^{-\int_t^T r_sds}  (V^{\pi}_T - F) - V^\pi_t \Big), \quad t\in [0,T]. 
\end{align}
Hence the strategy $\hat\pi$ minimizes the risk associated to the total price.
Clearly $Y^{\hat\pi} _T = Y_T = F$.
Here above, the set $\Pi_\MM$ denotes the admissible strategies.

\bigskip
Considering the risk measure \eqref{risk-measure}, we can write the problem \eqref{Y} in the following way:
\begin{align}
\label{Y3}
Y^{\hat\pi}_t =Y_t&=\essinf_{\pi\in\Pi_\MM}\: \esssup_{\Q\in\mathcal{Q}_\MM}{E}_{{\Q}}\big[- \big(e^{-\int_t^Tr_sds}V^{\pi}_T-V^{\pi}_t-e^{-\int_t^Tr_sds}F\big)|\mathcal{M}_t\big] \nonumber \\
&=\essinf_{\pi\in\Pi_\MM}\:  \esssup_{\Q \in\mathcal{Q}_\MM} {E}_{{\Q}}\bigg{[}
e^{-\int_t^Tr_sds}F-\int_t^Te^{-\int_t^sr_udu}\pi_s(\alpha_s-r_s)ds\nonumber\\
&\quad -\int_t^Te^{-\int_t^sr_udu}\pi_s\sigma_sdB_s\nonumber\\
&\quad -\int_t^T\int_{\Rr}e^{-\int_t^sr_udu}\pi_s\gamma_s(z)\tilde{H}(ds,dz)\bigg{|}\mathcal{M}_t\bigg{]} .
\end{align}

A solution to the problem \eqref{Y3} corresponds to finding $(\hat \pi, \hat\Q)\in \Pi_\MM \times \Qq_\MM$ such that
\begin{align}
\label{Ysolution}
Y_t = {E}_{\hat\Q}\bigg{[}
&e^{-\int_t^Tr_sds}F-\int_t^Te^{-\int_t^sr_udu}\hat\pi_s(\alpha_s-r_s)ds\nonumber\\
&-\int_t^Te^{-\int_t^sr_udu}\hat \pi_s\sigma_sdB_s\nonumber\\
&-\int_t^T\int_{\Rr}e^{-\int_t^sr_udu} \hat \pi_s\gamma_s(z)\tilde{H}(ds,dz)\bigg{|}\mathcal{M}_t\bigg{]}, \quad t\in [0,T].
\end{align}
The cost of the hedge is then $C^{\hat\pi} _t = E_{\hat\Q} \big[ e^{-\int_t^T r_sds} (F- V^{\hat\pi }_T) \vert \Mm_t \big]$.
To have a unique description of the optimal strategy, we set $E_{\hat\Q} [C^{\hat\pi} _0] =0 $ and $v= E_{\hat\Q} [Y^{\hat\pi} _0] $.

\bigskip

This kind of approach to hedging is treated in \cite{Delong2012} (see also \cite{Delong2013}) in the context of a financial market driven by a Brownian motion and an insurance payment process driven by a Poisson process. 
See also \cite{Balter2015} for a study of a similar problem within an insurance perspective, but Brownian driven dynamics.
We also refer to \cite{Karlsen2014}, where a first study of this problem is given in the context of a market driven by a Brownian motion and a doubly stochastic Poisson noise. 
Comparatively, in the present paper we consider a more general market model and a substantially different structure of admissible strategies. 

The admissible scenarios are described by a measure change via shift transformation. With respect to this, we suggest a version of the Girsanov theorem that explicitely deals with time-change. We note that, in this context, the measure change is not structure preserving in general.

The hedging problem \eqref{Y3}-\eqref{Ysolution} is tackled using backward stochastic differential equations (BSDEs) and stochastic differential games. 
Our study is carried through in the context of two different filtrations: $\MM=\FF$, which is substantially the information flow generated by the noises, and $\MM=\GG$, which is the filtration that, additionally to $\FF$, includes initial knowledge of the time-change process. These two settings lead to different BSDEs depending on their measurability properties.
We treat the solutions exploiting the martingale random field properties of the driving noises and, in the case of $\GG$, we also rely on the better explicit structure of the noise (which allows for a more explicit stochastic representation theorem).

In the case of information flow $\GG$, BSDEs driven by doubly stochastic L\'evy noises are treated in \cite{DS2014}. 
We also mention that these integral representation theorems are studied in \cite{Sjursen2011} with different approaches: via chaos expansions and via the non-anticipating derivative (see also \cite{DiNunno2007b} for a review on stochastic derivation). 
As for filtration $\FF$, we rely on the general results of \cite{Carbone2008}, which we adapt to the random field set-up.

Even though we can regard the information flow $\FF$ as partial with respect to $\GG$, the problems presented here are not the same as in the study on BSDEs with partial information, see e.g. \cite{Ceci} in the case of mean-variance hedging.

Finally, we remark that the hedging criteria we consider differs from mean-variance hedging in the objective function to minimize: mean-variance hedging identifies the strategy by minimizing the quadratic cost, see e.g. \cite{jeanblanc2011} in the context of prices modeled by general semimartingales, and \cite{lim2005} for the case of dynamics driven by a Brownian motion and doubly stochastic Poisson noises.

\bigskip
The paper is structured as follows:
the next section provides details about the framework and the BSDEs considered. In Section 3, we study shift transformations, while Section 4 is dedicated to the actual solution of the hedging problem in the two information flows considered.
Section 5 concludes with comments on the results obtained.

\section{The framework and preliminary results}

Hereafter, we give full detail of the noises considered in \eqref{stock} and the stochastic structures used. For this we refer to \cite{DS2014} and \cite{Sjursen2011}. In particular, we apply stochastic integration with respect to martingale random fields, see \cite{DiNunno2010} and \cite{Cairoli1975}.

\subsection{The random measures and their properties}

Let $(\Omega, \Ff, \Prob)$ be a complete probability space and $\X := [0,T] \times \RR$. We will consider $\X = \big ([0,T] \cup \{0\} \big) \cup \big( [0,T]\times \Rr\big)$, where $\Rr := \mathbb{R}\setminus \{0\}$ and $T>0$. 
Clearly, $[0,T] \cup \{0\} \simeq [0,T] $.
Denote $\Bb_\X$ the Borel $\sigma$-algebra on $\X$. Whenever we write $\Delta \subset \X$,  we intend a set $\Delta$ in $\Bb_\X$. 

The two dimensional stochastic process $\lambda:= (\lambda^B, \lambda^H)$ represents the intensity of the stochastic time distortion applied in the noise. Each component $\lambda^l$ for $l=B,H$, satisfies
\begin{enumerate}
\item $\lambda^l_t \geq 0$ $\Prob$-a.s. for all $t\in [0,T]$,
\label{item:lambda_1}
\item $\lim_{h\to 0} \mathbb{P} \big( \big| \lambda^l _{t+h} -\lambda^l_t \big| \geq \epsilon \big) = 0$ 
for all $\epsilon > 0$ and almost all $t \in \Tt$,
\label{item:lambda_2}
\item $\E \big[ \int_0^T  \lambda^l _t  \ins dt \big] < \infty. $ 
\label{item:lambda_3}
\end{enumerate}
Denote the space of all processes $\lambda := (\lambda^{B}, \lambda^{H})$ satisfying i), ii) and iii) by $\LL$. 


Correspondingly, we define the random measure $\Lambda$ on $\X$ by 
\begin{equation}
\Lambda(\Delta) := \int\limits_0^T \ind_\Delta (t, 0) \ins \lambda^B _tdt + \int\limits_0^T \int\limits_{\Rr} \ind_\Delta (t,z) \ins \nu(dz) \lambda^H_t dt, \quad \Delta \in \X.
\label{eq:measure_Lambda}
\end{equation}
Here $\nu$ is a (deterministic) $\sigma$-finite measure on the Borel sets of $\Rr$ satisfying 
\begin{equation*}
\int_\Rr z^2 \ins \nu(dz)<\infty. 
\end{equation*}
We denote the $\sigma$-algebra generated by the values of $\Lambda$  by $\Ff^\Lambda$.
Furthermore, $\Lambda^H$ denotes the restriction of $\Lambda$ to $[0,T]\times \Rr$ and $\Lambda^B$ the restriction of $\Lambda$ to $[0,T]\times \{0\}$. Hence $\Lambda(\Delta)=\Lambda^B(\Delta\cap[0,T]\times\{0\})+\Lambda^H(\Delta\cap[0,T]\times\Rr)$, $\Delta\subseteq \X$. 
Here below we introduce the noises driving the stochastic dynamics.

\begin{definition}
\label{definition:listA}
$B$ is a signed random measure on the Borel sets of $\Tt \times \{0\}$, satisfying
\begin{enumerate}
\renewcommand{\theenumi}{A\arabic{enumi})}
\item $\Prob\Big( B(\Delta) \leq x \,\Big| \Ff^\Lambda \Big) = \Prob\Big( B(\Delta) \leq x  \,\Big| \Lambda^B(\Delta) \Big) =  \Phi \big(\frac{x}{\sqrt{ \Lambda^B(\Delta)}}\big)$, $x\in\mathbb{R}$, $\Delta \subseteq [0,T]\times\{0\}$,
\label{list:A1}
\item $B(\Delta_1)$ and $B(\Delta_2)$ are conditionally independent given $\Ff^{\Lambda}$ whenever $\Delta_1$ and $\Delta_2$ are disjoint sets. 
\label{list:A2}
\end{enumerate}
Here $\Phi$ stands for the cumulative probability distribution function of a standard normal random variable.

$H$ is a random measure on the Borel sets of $\Tt\times \Rr$, satisfying
\begin{enumerate}
\renewcommand{\theenumi}{A\arabic{enumi})}
\setcounter{enumi}{2}
\item $\Prob\Big( H(\Delta) = k \,\Big|\Ff^\Lambda \Big) = \Prob\Big( H(\Delta) = k \,\Big| \Lambda^H(\Delta) \Big)  =  \frac{\Lambda^H(\Delta)^k}{k!} e^{-\Lambda^H(\Delta)}$, $k\in \mathbb{N}$, $\Delta \subseteq \Tt\times \Rr$,
\label{list:A3}
\item $H(\Delta_1)$ and $H(\Delta_2)$ are conditionally independent given $\Ff^{\Lambda}$ whenever $\Delta_1$ and $\Delta_2$ are disjoint sets. 
\label{list:A4}
\end{enumerate}

Furthermore, we assume that 
\begin{enumerate}
\renewcommand{\theenumi}{A\arabic{enumi})}
\setcounter{enumi}{4}
\item $B$ and $H$ are conditionally independent given $\Ff^\Lambda$.
\label{list:A5}
\end{enumerate}
\end{definition}

Substantially, conditional on  $\Lambda$, we have that $B$ is a Gaussian random measure and  $H$ is a Poisson random measure. 
We refer to \cite{Grigelions1975} or \cite{Kallenberg1997} for the existence of the above conditional distributions.

Let $\Ht:= H-\Lambda^H$ be the signed random measure given by
\begin{equation*}
\Ht(\Delta) = H(\Delta) - \Lambda^H(\Delta), \quad \Delta \subseteq \Tt\times \Rr.
\end{equation*}

\begin{definition}
We define the signed random measure $\mu$ on the Borel subsets of $\X$ by
\begin{equation}
\mu(\Delta) := B\Big(\Delta \cap[0,T]\times \{0\} \Big) + \Ht\Big(\Delta \cap [0,T]\times \Rr\Big), \quad \Delta \subseteq \X.
\label{eq:mu_definition}
\end{equation}
\end{definition}

The random measures $B$ and $H$ are related to a specific form of time-change for Brownian motion and pure jump L{\'e}vy process. More specifically define $B_t := B( [0,t]\times \{ 0\})$, $\Lambda^B_t := \int_0^t \lambda^B _s \ins ds$, $\eta_t := \int_0^t\int_{\Rr} z \ins \Ht(ds,dz)$ and $\hat{\Lambda}^H_t := \int_0^t \lambda^H _s \ins ds$, for $t \in [0,T]$.

We can immediately see the role that the time-change processes $\Lambda^B$ and $\hat{\Lambda}^H$ play by studying the characteristic function of $B$ and $\eta$. In fact, from \ref{list:A1} and \ref{list:A3} we see that the conditional characteristic functions of $B_t$ and $\eta_t$ are given by
\begin{align}
\E \big[ e^{ic B_t} \big| \Ff^\Lambda \big] &= \exp\bigg\{ \int\limits_0^t \frac{1}{2}c^2  \ins \lambda^B _s ds \bigg\} = \exp\bigg\{ \frac{1}{2}c^2 \Lambda_t^B \bigg\}, \quad c\in \RR,
\label{eq:characteristic_B} \\
\E \big[ e^{ic \eta_t} \big| \Ff^\Lambda \big] &= \exp\bigg\{\int\limits_0^t \int\limits_\Rr \big[e^{icz} -1-icz \big]\ins \nu(dz) \lambda^H_s ds \bigg\} \nonumber \\
&= \exp\bigg\{ \bigg( \int\limits_\Rr \big[e^{icz} -1-icz \big]\ins \nu(dz) \bigg) \, \hat{\Lambda}_t^H \bigg\}, \quad c \in \RR.
\label{eq:characteristic_H} 
\end{align}

Indeed, there is a strong connection between the distributions of $B$ and the Brownian motion, and between $\eta$ and a centered pure jump L{\'e}vy process with the same jump behavior. The relationship is based on a random distortion of the time scale. The following characterization is due to \cite[Theorem 3.1]{Serfozo1972} (see also \cite{Grigelions1975}). 
\begin{theorem}
Let $W_t$, $t\in \Tt$, be a Brownian motion and $N_t$, $t\in\Tt$, be a centered pure jump L\'evy process with Levy measure $\nu$. Assume that both $W$ and $N$ are independent of $\Lambda$. 
Then $B$ satisfies \ref{list:A1}-\eqref{eq:characteristic_B} and \ref{list:A2} if and only if, for any $t\geq 0$,
\begin{equation*}
B_t  \stackrel{d}{=} W_{\Lambda_t^B},
\end{equation*}
and $\eta$ satisfies \ref{list:A3}-\eqref{eq:characteristic_H} and \ref{list:A4} if and only if, for any $t\geq 0$,
\begin{equation*}
\eta_t \stackrel{d}{=} N_{\hat{\Lambda}^H_t}. 
\end{equation*}
\end{theorem}
In addition, $B$ is infinitely divisible if $\Lambda^B$ is infinitely divisible, and $\eta$ is infinitely divisible if $\hat{\Lambda}^H$ is infinitely divisible, see \cite[Theorem 7.1]{Barndorff-Nielsen2006}.

\subsection{Stochastic non-anticipating integration and representation theorems}
\label{section:integration}

Let us define $\FF^{\mu}= \{\Ff_t^{\mu}$,\; $t\in [0,T]\}$ as the filtration generated by $\mu(\Delta)$, $\Delta\subseteq [0,t]\times \RR$, $t \in [0,T]$. In view of \eqref{eq:mu_definition}, \ref{list:A1} and \ref{list:A3}, we can see that, for any $t\in [0,T]$, 
\begin{equation*}
\Ff_t^\mu = \Ff_t^B \vee \Ff_t^H \vee \Ff_t^\Lambda, 
\end{equation*}
where $\Ff_t^B$ is generated by $B(\Delta\cap [0,T]\times \{0\})$, $\Ff_t^H$ by $H(\Delta\cap [0,T]\times \Rr)$, and $\Ff^\Lambda_t$ by $\Lambda(\Delta)$, $\Delta \in [0,t]\times \RR$. This is an application of \cite[Theorem 1]{Winkel2001} and \cite[Theorem 2.8]{Sjursen2011}. Set $\FF= \{\Ff_t$,\; $t\in [0,T]\}$, where
\begin{equation*}
\Ff_t := \bigcap_{r>t} \Ff^{\mu}_r .
\end{equation*}
Furthermore, we set $\GG=\{\Gg_t,\; t\in [0,T]\}$ where $\Gg_t:=\Ff_t \vee \Ff^{\Lambda}$. 
Remark that $\Gg_T = \Ff_T$, $\Gg_0=\Ff^\Lambda$, while $\Ff_0$ is trivial. From now on we set $\Ff=\Ff_T$.

\begin{lemma}
The filtration $\GG$ is right-continuous. 
\end{lemma}
\begin{proof}
This can be shown adapting classical arguments for the L\'evy case as in e.g. \cite[Theorem 2.1.9]{Applebaum2004}.
\end{proof}

For $\Delta \subset (t,T]\times \RR$, the conditional independence \ref{list:A2} and \ref{list:A4} means that
\begin{equation}
\E\big[ \mu(\Delta) \ins \big| \Gg_t \big] = \E\big[ \mu(\Delta) \ins \big| \Ff_t \vee \Ff^\Lambda \big] = \E\big[ \mu(\Delta) \ins \big| \Ff^\Lambda \big] =0.
\label{eq:conditional_independence_explained}
\end{equation}

Hence, $\mu$ is a {\it martingale random field (with conditional orthogonal values in $L^2(\Omega,\Ff, \Prob)$} with respect to $\GG$ in the sense of \cite{DiNunno2010} (see Definition 2.1), since 
\begin{itemize}
\item $\mu$ has a $\sigma$-finite variance measure 
\begin{equation*}
m(\Delta) := \E \big[ \mu(\Delta)^2] =  \E \big[ \Lambda(\Delta)], \quad \Delta\subseteq \X,
\end{equation*}
with $m(\{0\}\times \R) = 0$,
\item
it is additive on pairwise disjoint sets in $\Bb_\X$ and $\sigma$-additive with convergence in $L^2$,
\item $\mu$ is $\GG$-adapted, 
\item
it has the martingale property \eqref{eq:conditional_independence_explained},
\item $\mu$ has conditionally orthogonal values, if $\Delta_1, \Delta_2 \subset (t,T]\times \RR$ such that $\Delta_1 \cap \Delta_2 = \emptyset$ then, combining \ref{list:A2}, \ref{list:A4}, \ref{list:A5} and \eqref{eq:conditional_independence_explained},
\begin{equation*}
\E \Big[ \mu(\Delta_1) \mu(\Delta_2) \ins \Big| \Gg_t \Big] = \E \Big[ \mu(\Delta_1) \ins\Big| \Ff^\Lambda \Big] \E \Big[ \mu(\Delta_2)\ins \Big| \Ff^\Lambda \Big] =0.
\label{eq_mu:conditionally_orthogonal_values} 
\end{equation*}
\end{itemize}
In \cite{Cairoli1975} there is a discussion about martingale (difference) random fields and the role of ordering associated with the information flow. In their terminology the martingale random fields here treated is both a ``strong" and a `weak" martingale.

\bigskip
Denote $\II_\GG$ as the subspace of $L^2(\Tt\times \RR \times \Omega, \Bb_{\X} \times \Ff, \Lambda\times \Prob)$ of the random fields admitting a $\GG$-predictable modification, in particular 
\begin{equation}
\| \phi \|_{\II_\GG} :=  \Bigg( \E \Bigg[ \int\limits_0^T \phi _s(0)^2 \ins \lambda^B_s ds + \int\limits_0^T\int\limits_\Rr \phi _s(z)^2 \ins \nu(dz) \lambda^H _s ds \Bigg] \Bigg) ^{\frac{1}{2}} < \infty.
\label{eq:II_defined}
\end{equation}
For any $\phi \in \II_\GG$, we define the (It\^o type) non-anticipative stochastic integral $I: \II_\GG \Rightarrow L^2(\Omega, \Ff, \Prob)$ by
\begin{equation*}
I(\phi ) := \int\limits_0^T \phi _s(0) \ins dB_s + \int\limits_0^T\int\limits_\Rr \phi_s(z) \ins \Ht(ds,dz).
\end{equation*}
We refer to \cite{DiNunno2010} for the details on the integration with respect to martingale random fields of the type discussed here. Recall that $I$ is a linear isometric operator:
\begin{equation*}
\sqrt{ E\big[ I(\phi)^2 \big] } = \|I(\phi) \|_{L^2(\Omega, \Ff, \Prob)} = \| \phi \|_{\II_\GG}.
\label{eq:isometry_general}
\end{equation*}

Because of the structure of the filtration considered, we have the following result (see \cite{DS2014}):
\begin{lemma}
Consider $\xi \in L^2\big(\Omega,\mathcal{F}^\Lambda,\mathbb{P}\big)$ and $\phi \in \mathcal{I}_\GG$. Then
\begin{equation*}
\xi I(\phi) = I( \xi\phi),
\end{equation*}
whenever either side of the equality exists as an element in $L^2\big(\Omega,\mathcal{F},\mathbb{P}\big)$.
\label{lemma:f_alpha_and integration}
\end{lemma}


\begin{remark}
It is easy to see that the random field $\mu$ is also a martingale random field with respect to $\FF$ and the non-anticipating integration can be done also with respect to $\FF$ as for $\GG$. 
We denote $\II_\FF$ the corresponding set of integrands. However, results such as Lemma \ref{lemma:f_alpha_and integration} and the forthcoming representation would not hold. See also \cite[Remark 4.4]{Sjursen2011}. 

We remark that $\Ff^\mu_t := \sigma\{ \mu(\Delta), \; \Delta \subseteq [0,t]\times \R\} = \sigma\{I(\phi 1_{\Delta}), \Delta \subseteq [0,t] \times \R),\; \phi\in\II_\FF\}$ (indeed $\mu(\Delta)=I(\ind_{\Delta})$) and
$\Gg_t := \sigma\{ \mu(\Delta), \; \Delta \subseteq [0,t] \times \R; \Lambda(\Delta), \Delta \subseteq [0,T] \times \R\} = \sigma\{I(\phi 1_{\Delta}), \Delta \subseteq [0,t] \times \R),\; \phi\in\II_\GG\}$
\label{remark:about_F}
\end{remark}

The following representation theorems are given in \cite{DS2014}.

\begin{theorem}
{\bf Integral representation theorem.}
Assume $\xi \in L^2\big(\Omega,\Ff,\mathbb{P}\big)$. Then there exists a unique $\phi \in \II_\GG$ such that 
\begin{equation}
\xi = \E\big[ \xi \ins \big| \Ff^{\Lambda} \big] + \int\limits_0^T\int\limits_\RR \phi _s(z) \ins \mu(ds,dz).
\label{eq:ito_representation}
\end{equation}
\label{theorem:ito_representation}
\end{theorem}
\noindent 
Note that the two summands in \eqref{eq:ito_representation} are orthogonal. Here $\E[ \xi \ins| \Ff^{\Lambda}]$ represents the stochastic component of $\xi$ that cannot be recovered by integration on $\mu$.

\begin{remark}
The existence of such a representation is treated in \cite[Chapter 3]{Jacod2003}, where the result is obtained after a discussion on the solution of \emph{the martingale problem}.
In \cite{DS2014}, the existence and uniqueness of the above representation is proved by classical density arguments inspired by \cite[Section 4]{Oksendal2005} and \cite{Lokka2005}. 
In \cite{Sjursen2011}, the representation is given with respect to $\Ht$ using orthogonal polynomials. There, an \emph{explicit formula} for the integrand $\phi$ is derived by means of the non-anticipating derivative with respect to $\GG$, see \cite[Theorem 5.1]{Sjursen2011}. This result holds for more general choices of $\Lambda^H$, but with an assumption on the moments.
The non-anticipating derivative is well-definied with respect to any martingale random-field with orthogonal values and is an operator on the whole $L^2(\Omega,\F,\Prob)$. The random variable $\xi^0 = E[\xi \vert \F^\Lambda_T]$ is characterized by having non-anticipating derivative identically null.

There are other related results in the literature, e.g. in \cite[Proposition 41]{Yablonski2007} the same representation is proved for a class of Malliavin differentiable random variables (Clark-Ocone type results).

If an $\Ff^H_T$-measurable $\xi$ is considered, then an integral representation is given
in the general context of (marked) point processes, see for instance \cite[Theorem 4.12 and 8.8]{Bremaud1981} or \cite{Davis1976,Boel1975,Jacod1975}. Theorem \ref{theorem:ito_representation} differs in the choice of filtration, which also leads to different integrals. In \cite{Bremaud1981,Davis1976,Boel1975,Jacod1975}, the integrator in the representation theorem is given by $H-\vartheta$, where $\vartheta$ is $\mathbb{F}^H$-predictable compensator of $H$. Our $\Lambda^H$ is not $\mathbb{F}^H$-predictable. 
\label{remark:integral_representations}
\end{remark}

\begin{theorem}
{\bf Martingale representation theorem.}
Assume $M_t$, $t\in [0,T]$, is a $\GG$-martingale. Then there exists a unique $\phi \in \II_\GG$ such that 
\begin{equation*}
M_t = \E\big[ M_T \ins \big| \Ff^\Lambda \big] + \int\limits_0^t \int\limits_\RR \phi _s(z) \ins \mu(ds,dz), \quad t\in [0,T].
\end{equation*}
\label{Teorem:G_martingales}
\end{theorem}

We observe that, in the case we consider $\mu$ to be a martingale random field with respect to $\FF$, the corresponding results take a different form.
See \cite{DiNunno2010}.
In particular, we have:

\begin{theorem}
\label{thm:ito_representation2}
{\bf Integral representation theorem.}
Assume $\xi \in L^2\big(\Omega,\Ff,\mathbb{P}\big)$. Then there exists a unique $\phi \in \II_\FF$ such that 
\begin{equation}
\xi = \xi^0 + \int\limits_0^T\int\limits_\RR \phi _s(z) \ins \mu(ds,dz),
\label{eq:ito_representation2}
\end{equation}
where $\xi^0$ is a random variable in $L^2\big(\Omega,\Ff,\mathbb{P}\big)$ orthogonal to the integral part.
\label{theorem:ito_representation2}
\end{theorem}
In terms of the non-anticipating derivative with respect to $\FF$ as studied in \cite{DiNunno2010}, the random variable $\xi^0$ is characterised by having derivative identically null.

\subsection{Backward stochastic differential equations driven by $\mu$}

The problem of hedging considered in this paper leads to different types of BSDEs depending on the information considered. Hereafter, we give an overview of the results needed in the sequel related to both types. In particular, the comparison theorems will play a central role in the solution of the optimisation problem \eqref{Y3}.
Our references are \cite{DS2014} and \cite{Carbone2008}.

\bigskip
\bigskip
\subsubsection{Information flow $\GG$}
In the case of information flow $\GG$, the BSDE of reference is of the form:
\begin{equation}
Y_t = \xi + \int\limits_t^T g_s\big(\lambda_s,Y_s,\phi_s \big) \ins ds - \int\limits_t^T\int\limits_{\mathbb{R}} \phi_s(z)  \ins \mu(ds,dz), \quad t\in[0,T].
\label{eq:defbsde}
\end{equation}
Given a terminal condition $\xi$ and a driver (or generator) $g$, a  solution is given by the couple of $\GG$-adapted processes $(Y,\phi)$ on $(\Omega, \Ff, \Prob)$ satisfying the equation above. Hereafter, we characterise explicitly the functional spaces in use and the elements of the BSDE to obtain a solution.

Let $S_\GG$ be the space of $\GG$-adapted stochastic processes $Y_t(\omega)$, $t\in [0,T]$, $\omega \in \Omega$, such that 
\begin{equation*}
\| Y \|_{S_\GG} := \sqrt{ \E \bigg[ \sup_{t\in [0, T]} |Y_t|^2 \bigg] } < \infty,
\end{equation*}
and let $\Hg$ be the space of $\GG$-predictable stochastic processes $f _t(\omega)$, $t\in [0,T]$, $\omega \in \Omega$, such that
\begin{equation*}
\E \Bigg[ \int\limits_0^T f_s ^2 \ins ds \Bigg] < \infty.
\end{equation*}
%
Denote $\Phi$ the space of functions $\phi :\RR \to \RR$ such that 
\begin{equation*}
| \phi(0)|^2+ \int\limits_\Rr \phi(z)^2 \ins \nu(dz) < \infty,
\label{eq:Phi}
\end{equation*}
where $\nu$ is the jump measure of the market dynamics.
\begin{defn}
We say that $(\xi,g)$ are \emph{standard parameters} when $\xi\in L^2\big(\Omega,\Ff,\mathbb{P}\big)$ and $g: [0,T]\times [0,\infty)^2 \times \mathbb{R}\times \Phi  \times \Omega \to \mathbb{R}$ such that $g$ satisfies the following conditions:
\begin{itemize}
\item
$ g_\cdot(\lambda, Y,\phi,\cdot) \text{is $\GG$-adapted for all $\lambda\in\LL$, $Y\in S_\GG$, $\phi\in\II_\GG$,}$ \label{eq:f_cond0} 
\item
$ g_\cdot(\lambda_{\cdot} ,0,0, \cdot )  \in \Hg, \text{ for all }\lambda\in\LL$ \label{eq:f_cond1} 
\item
there exists $K_g>0$ for which
\begin{align*}
\big|  &g_t \big( (\lambda^B, \lambda^H), y^{(1)}, \phi^{(1)} \big) - g_t \big((\lambda^B,\lambda^H), y^{(2)}, \phi^{(2)} \big) \big|  \leq K_g \Big( \big| y^{(1)}-y^{(2)} \big|   \\
&+\big| \phi^{(1)}(0) - \phi^{(2)}(0) \big|\sqrt{\lambda^B}   + \sqrt{ \int\limits_\Rr  | \phi^{(1)}(z) - \phi^{(2)}(z) |  ^2 \ins  \nu(dz) } \sqrt{\lambda^H} \Big), 
\end{align*} 
\item
$ \text{for all }(\lambda^B,\lambda^H) \in [0,\infty)^2, y^{(1)},y^{(2)} \in \RR, \text{ and } \phi^{(1)},\phi^{(2)} \in \Phi \; dt\times d\Prob \text { a.e.}$  \nonumber 
\end{itemize}
\label{def:standard_parameters}
\end{defn}

\begin{theorem}
Let $(g,\xi)$ be standard parameters. Then  
there exists a unique couple $(Y,\phi) \in S_\GG \times \II_\GG$ such that
\begin{align}
Y_t &= \xi + \int\limits_t^T g_s\big(\lambda_s,Y_s,\phi_s \big) \ins ds - \int\limits_t^T\int\limits_\RR \phi_s(z) \ins \mu(ds,dz) \nonumber \\
&= \xi + \int\limits_t^T g_s\big(\lambda_s,Y_s,\phi_s \big) \ins ds - \int\limits_t^T \phi_s(0)\ins dB_s - \int\limits_t^T  \int\limits_{\Rr} \phi_s(z) \ins \Ht(ds,dz).
\label{eq:bsde_theorem}
\end{align}
\label{teorem:BSDE_existence_uniqueness}
\end{theorem}
\begin{remark} The initial point $Y_0$ of the solution $Y$ is \emph{not} necessarily a (deterministic) constant. From the definition of $\GG$ and \eqref{eq:bsde_theorem}, we see that $Y_0$ is a square integrable $\Ff^\Lambda$-measurable random variable. 
To be specific, we have:
\begin{align*}
Y_0 &= \E \Big[ \xi + \int\limits_0^T g_s\big(\lambda_s,Y_s,\phi_s \big) \ins ds - \int\limits_0^T \phi_s(0)\ins dB_s - \int\limits_0^T  \int\limits_{\Rr} \phi_s(z) \ins \Ht(ds,dz) \ins \Big| \Ff^\Lambda \Big]\\
&= \E \Big[ \xi + \int\limits_0^T g_s(\lambda_s,Y_s,\phi_s) \ins ds \ins \Big| \Ff^\Lambda \Big].
\end{align*}
\end{remark}
For a linear BSDE of the form \eqref{eq:linearBSDE}, there exists an explicit representation of the solution.
\begin{theorem}
\label{linearGG}
Assume we have the following BSDE:
\begin{align}
-dY_t =&\; \Big[ A_tY_t + C_t+ E_t(0) \phi_t(0) \sqrt{\lambda_t^B} + \int\limits_\Rr E_t(z) \phi_t (z) \ins \nu(dz) \sqrt{ \lambda_t^H} \Big] \ins dt  \nonumber \\
 &-\phi_t(0) \ins dB_t   -\int\limits_\Rr \phi_t(z) \ins \tilde{H}(dt,dz), \quad Y_T= \xi,
\label{eq:linearBSDE}
\end{align}
where the coefficients satisfy
\begin{enumerate}
\renewcommand{\theenumi}{\roman{enumi})}
\renewcommand{\labelenumi}{\theenumi}
\item $A$ is a bounded stochastic process, there exists $K_A > 0$ such that $|A_t | \leq K_A$ for all $t\in[0,T]$ $\Prob$-a.s.,
\label{list:boundedAB}
\item $C\in \Hg$, 
\label{list:CinH2}
\item $E \in \II_\GG$,
\item There exists a $K_E > 0$ such that $ 0 \leq E_t(z) < K_E z $ for $z\in\Rr$, and $| E_t(0) | < K_E$ $\;dt\times d \Prob$-a.e.
\label{list:finite_E_integral1}
\end{enumerate}
Then \eqref{eq:linearBSDE} has a unique solution $(Y,\phi)$ in $S_\GG \times \II_\GG$ and $Y$ has representation
\begin{equation*}
Y_t = \E \Big[ \xi \Gamma_T^t + \int\limits_t^T \Gamma^s_t C_s \ins ds \ins \Big| \Gg_t \Big], \quad t\in[0,T],
\end{equation*}
where
\begin{align*}
\Gamma_s^t :=&\; \exp\Big\{ \int\limits_t^s A_u -\frac{1}{2} E_u(0)^2 \ind_{\{\lambda_u^B\neq 0\}} \ins du + \int\limits_t^s E_u(0) \frac{\ind_{\{\lambda_u^B\neq 0\}} }{\sqrt{\lambda_u^B}} \ins dB_u \\
& + \int\limits_t^s \int\limits_\Rr \Big[ \ln\big(1+E_u(z)\frac{\ind_{\{\lambda_u^H\neq 0\}} }{\sqrt{\lambda_u^H}}  \big) - E_u(z)\frac{\ind_{\{\lambda_u^H \neq 0\}} }{\sqrt{\lambda_u^H}} \Big] \ins \nu(dz)\lambda^H_u \ins du \\
& + \int\limits_t^s \int\limits_\Rr \ln\big(1+E_u(z)\frac{\ind_{\{\lambda_u^H\neq 0\}} }{\sqrt{\lambda_u^H}} \big)  \ins \tilde{H}(du,dz) \Big\}, \quad s \geq t.
\end{align*}
\label{teorem:theorem_linear_BSDEs}
\end{theorem}
Note that $\Gamma^s_t = \frac{\Gamma_s^0}{\Gamma_t^0}$.

\bigskip
The next result is a comparison theorem. This result is crucial in the solution of the optimisation problem \eqref{Y3}.

\begin{theorem}  {\bf Comparison theorem.}
\label{thm:comparison}
Let $(g^{(1)},\xi^{(1)})$ and $(g^{(2)},\xi^{(2)})$ be two sets of standard parameters for the BSDE's with solutions $(Y^{(1)},\phi^{(1)})$, $(Y^{(2)},\phi^{(2)})\in S_\GG \times \II_\GG$. Assume that
\begin{equation*}
g^{(2)}_t(\lambda,y,\phi,\omega) = f_t\Big(y, \phi(0) \kappa (0) \sqrt{\lambda^B}  ,\int\limits_\Rr \phi(z) \kappa (z) \ins \nu(dz) \sqrt{\lambda^H},\omega  \Big),
\end{equation*}
where $\kappa \in \II_\GG$ satisfies condition \ref{list:finite_E_integral1} from Theorem \ref{teorem:theorem_linear_BSDEs} and $f$ is a function $f: [0,T] \times \mathbb{R} \times \mathbb{R} \times \mathbb{R}\times \Omega \to \mathbb{R}$ which satisfies, for some $K_f>0$,
\begin{equation*}
|f_t(y,b,h) - f_t(y',b',h')|  \leq K_f\Big( |y-y'| + |b-b'| + |h-h'| \Big), \label{eq:h_lipschitz} \\
\end{equation*}
$dt\times d\Prob$ a.e. and 
\begin{equation*}
\E \Big[ \int\limits_0^T |f_t(0,0,0)|^2 \ins dt \Big] < \infty. \end{equation*}
If $\xi^{(1)} \leq \xi^{(2)}$ $\Prob$-a.s. and $g^{(1)}_s(\lambda_s,Y^{(1)}_s,\phi^{(1)}_s ) \leq g^{(2)}_s(\lambda_s,Y^{(1)}_s,\phi^{(1)}_s )$ $dt\times d\Prob$-a.e., then
\begin{equation*}
Y^{(1)}_t \leq Y^{(2)}_t  \quad dt\times d\Prob \text{-a.e.}
\end{equation*}
\end{theorem}

\bigskip
\bigskip
\subsubsection{Information flow $\FF$}
In the case of information flow $\FF$ the BSDE of reference takes the form:
\begin{align}
Y_t=\xi+\int_t^T \int_\R f_s(Y_s, \phi_s(z) )\langle \mu\rangle (ds,dz) -\int_t^T \int_\R \phi_s(z) \mu (ds,dz) -N_T+N_t,
\label{eq:bsdegen}
\end{align}
where
\begin{enumerate}[(i)]
\item $\mu(dt,dz)$, $t\in [0,T]$, $z\in \R$, is the $(\FF,\Prob)$-martingale random field \eqref{eq:mu_definition},
\item $\langle \mu\rangle (dt,dz)$ is its conditional variance measure, see \cite{DiNunno2010} Theorem 2.1, which is in fact the correspondent to the conditional quadratic variation for martingales, see \cite{Protter2005}, and
\item $N$, with $N_0=0$, is a square integrable $(\FF,\Prob)-$martingale orthogonal to $\mu$, i.e. for every set $A\in\Bb_\R$, for $\mu_t(A):= \mu((0,t]\times A)$, $t \in [0,T]$, the quadratic variation $[ N, \mu(A) ]$ is a uniformly integrable martingale.
\end{enumerate}
Moreover, we have
\[
\langle \mu (A) \rangle_t= \int_0^t\int_A \lambda^B_s \delta_{\{0\}}(dz)ds+  \int_0^t\int_A {1}_{{\R}_0}(z)\nu(dz)\lambda^H _sds.
\]

The existence and uniqueness of the solution of \eqref{eq:bsdegen} is treated adapting Proposition 2.1 and Lemma 2.2 of \cite{Carbone2008} to the martingale random field case.
Here we present the variation of these results in the form used later.
The definition of standard parameters is analogous to Definition \ref{def:standard_parameters}, but referred to $\FF$. The same is intended for the spaces involved. 
BSDEs of the type \eqref{eq:bsdegen} with standard parameters admit a unique solution, which is characterized by the triple $(Y,\phi,N)$.
\begin{lem}\label{lem:carbonelemma}
Let $a,b,c$ be $\FF$-predictable random fields with $a$ bounded and
$$
\E \Big[ \int_0^T\int_\R b^2_s(z) \langle \mu \rangle (ds,dz) \Big] < \infty.
$$
Let $\mathcal{E}$ be the Dol\'eans exponential of the martingale random field  $\int_0^t \int_\R b_s(z) \mu(ds,dz)$, $t\in [0,T]$, and define
\begin{align*}
\psi_t :=\exp\bigg(\int_0^t \int_\R a_s(z) \langle \mu\rangle (ds,dz) \bigg),\quad\quad\Psi_t=\psi_t\mathcal{E}_t, \quad t \in [0,T].
\end{align*}
Suppose that 
\begin{enumerate}[(i)]
\item $\mathcal{E}$ is a positive uniformly integrable martingale;
\item $\E[(\sup_{t\in[0,T]}\psi_t)^2\mathcal{E}^2_T]<\infty$,
\item $\E \big[ \int_0^T\int_\R \Psi_s(z) c_s(z) \langle \mu \rangle (ds,dz) \big] < \infty.$
\end{enumerate}
If the linear backward equation
\begin{align}
dY_t&=- \int_\R \bigg( a_t(z) Y_t+b_t(z) \phi_t(z) +c_t(z) \bigg) \langle \mu\rangle(ds,dz)+\int_\R\phi_t(z) \mu (dt,dz)+dN_t\nonumber\\
Y_T&=\xi,\label{eq:newlinearbsde}
\end{align}
has solution $(Y,\phi,N)$ in $S_\FF \times\mathcal{I}_{\FF}\times\mathcal{L}_\FF^{2,b}$, where $\mathcal{L}_\FF^{2,b}$ is the space of $L^2-$bounded $(\FF,\Prob)-$martingales, then $Y$ is given by
\begin{align}
Y_t=\E\bigg[\xi\frac{\Psi_T}{\Psi_t}+\int_t^T \int_\R \frac{\Psi_s}{\Psi_t}c_s(z) \langle \mu\rangle (ds,dz) \bigg|\mathcal{F}_t\bigg],\quad0\leq t\leq T.\label{eq:expbsde}
\end{align}
\end{lem}
\begin{proof}
The result presents weaker assumptions on the coefficients compared to Lemma 2.2 in \cite{Carbone2008} which serve better the applications to come. However, the proof follows substantially the same arguments and it will not be detailed here.
\end{proof}

Recall that the Dol\'ean exponential is positive if the martingale has jumps grater than $-1$. Conditions for uniform integrability can be found in e.g. \cite{LM78}.

\begin{remark}
\label{linearFF}
Whenever expression (\ref{eq:expbsde}) makes sense, it is a solution of the linear equation (\ref{eq:newlinearbsde}). 
\end{remark}

\bigskip
The next result is a comparison theorem for equations of the type \eqref{eq:bsdegen}.

\begin{theorem}
\label{thm:carbonethm}{\bf Comparison theorem.}
Consider two linear BSDEs of the form (\ref{eq:bsdegen}):
\[
Y^{(i)}_t=\xi^{(i)}+\int_t^T \int_\R f^{(i)}_s(Y^{(i)}_s, \phi^{(i)}_s)\langle \mu\rangle(ds,dz)-\int_t^T\int_\R \phi^{(i)}_s(z) \mu(ds,dz) -N^{(i)}_T+N^{(i)}_t,
\]
for $i=1,2$. Define $\delta Y_t :=Y^{(2)}_t-Y^{(1)}_t$, $\delta \phi_t(z) :=\phi^{(2)}_t(z) -\phi^{(1)}_t(z)$, and 
\begin{align*}
a_t(z) &:= \frac{f^{(2)}_t(Y^{(2)}_t,\phi^{(2)}_t(z) )-f^{(2)}_t(Y^{(1)}_t,\phi^{(2)}_t(z))}{\delta Y_t} {1}_{\delta Y_t\neq0},\\
b_t(z) & :=\frac{f^{(2)}_t(Y^{(1)}_t,\phi^{(2)}_t(z)) - f^{(2)}_t(Y^{(1)}_t,\phi^{1}_t(z))}{ \delta \phi_t(z) }{1}_{\delta \phi_t(z)\neq0},\\
c_t(z) &:= f^{(2)}_t(Y^{(1)}_t, \phi^{(1)}_t(z))-f^{(1)}_t(Y^{(1)}_t, \phi^{(1)}_t(z)).
\end{align*}
Then the process $\delta Y$ verifies the linear BSDE:
\begin{align*}
-d \,\delta Y_t &= \int_\R \big[a_t(z) \delta Y_t + b_t(z) \delta \phi_t(z) + c_t(z) \big] \langle \mu \rangle (dt,dz) 
+ \int_\R \delta \phi_t(z) \mu(dt,dz) - d\,\delta N_t,\\
\delta Y_T &= \xi^{(1)} - \xi^{(1)}.
\end{align*}
Assume that $a$ and $b$ verify condition (i) and (ii) in Lemma \ref{lem:carbonelemma}. Assume also that $\xi^{(2)}\geq \xi^{(1)}$ and, for any $t$, 
$c_t(z) \geq0\;\Prob-$a.s. Then, for any $t$, $Y^{(2)}_t\geq Y^{(1)}_t\;\Prob-$a.s.
\end{theorem}

\bigskip
\bigskip
\section{Change of measure: shift transformations}
The various possible scenarios considered are given by all probability measures $\Q$ equivalent to $\Prob$ obtained by shift transformation, see \eqref{risk-measure}. 
This is chosen as a feasible set of transformations that allow for an explicit evaluation of the Radon-Nikodym density. Hereafter, we study such transformations. We observe that, in the presence of time-change, such transformations do not lead to a self-preserving structure.
As illustration we can see that the doubly stochastic Poisson random measure will not be of such structure after the measure change.
We also recall that such shift transformations, when applied to L\'evy noises, are actually structure preserving.

\bigskip
We introduce the Radon-Nikodym density process $Z_t= Z^\theta_t$, $t\in [0,T]$, by
\begin{align}
&dZ^\theta_t=Z^\theta_{t^-} \Big( \theta^B_tdB_t+\int_{\Rr}\theta^H_t(z)\tilde{H}(dt,dz) \Big),\label{eq:radoniko}\\
&Z^\theta_0=1,\nonumber
\end{align}
for $\theta\in\mathcal{I}_\GG$ with $\theta_t(0) = \theta_t^B$ and
$\theta_t(z) = \theta_t^H(z)$, $z \in \R_0$, and $\theta^H_t(z)>-1$ $ \Prob\times\Lambda-a.e.$

An explicit strong solution of (\ref{eq:radoniko}) is obtained by application of the It\^o formula:
\begin{align*}\label{eq:density}
Z^\theta_t&=\exp\bigg{(}\int_0^t\theta^B_sdB_s-\int_0^t\frac{1}{2}(\theta^B_s)^2 \lambda^B_s ds\nonumber\\
&\quad+\int_0^t\int_{\Rr}\bigg[\ln\bigg(1+\theta^H_s(z)\bigg)-\theta^H_s(z)\bigg]\nu(dz)\lambda^H_sds \nonumber\\
&\quad+\int_0^t\int_{\Rr}\ln\bigg(1+\theta^H_s(z)\bigg)\tilde{H}(ds,dz)\bigg{)}, \quad t \in [0,T].
\end{align*}
Since we have assumed $\theta^H_t(z)>-1\;\; \Prob\times\Lambda-$a.e., we know that $\ln\bigg(1+\theta^H_s(z)\bigg)$ and the stochastic integration are well-defined. 
A generalized version of the Novikov condition, see \cite{LM78}, ensures uniform integrability of $Z$.

\begin{theorem}
\label{thm:girsanov}
Let $B$ and $\tilde{H}$ be as in Definition \ref{definition:listA} with respect to $\Prob$. 
Assume that $Z_t=Z^\theta_t$, $t\in [0,T]$, with $\theta\in\mathcal{I}_{\GG}$, is a positive uniformly integrable $(\GG,\Prob)-$martingale with $\E[Z_T^2]<\infty$, and define the probability measure $\Q$, equivalent to $\Prob$, by
\[
\frac{d\Q}{d\Prob}=Z_T.
\]
Define $B^{\theta}$ and $\tilde{H}^{\theta}$ by the dynamics
\begin{align*}
dB^{\theta}_t&:=dB_t-\theta^B_t d\Lambda^B_t ,\\
\tilde{H}^{\theta}(dt,dz)&:=\tilde{H}(dt,dz)-\theta^H_t(z)\Lambda^H (dt,dz),
\end{align*}
where we recall that $\Lambda^B(dt, \{0\}) = d\Lambda^B_t = \lambda^B_tdt$ and
$ \Lambda^H(dt,dz) = \nu(dz) \lambda^H_tdt$.
Moreover, for any bounded predictable $\psi$ such that $\int_0^T\int_{\Rr}\psi_t(z) \Lambda^H(dt,dz) < \infty,$ $ \Prob-$a.s., define the process 
\[
M^\theta _t(\psi):=\int_0^t\int_{\Rr}\psi_s(z)\tilde{H}^{\theta}(ds,dz),\quad0\leq t\leq T.
\]

Then $B^{\theta}$ is a continuous $(\GG,\Q)$-martingale and a time-changed $(\GG,\Q)$-Brownian motion. Also, $M^\theta (\psi)$ is a $(\GG,\Q)$-martingale, where $\tilde{H}^{\theta}$ is a $(\GG,\Q)$-martingale random field.

Moreover, if
\begin{align}
\E_{\Q}\bigg[\sup_{t\in[0,T]}\;|[B^{\theta}, M^\theta(\psi)]_t|\;\bigg]<\infty,\label{eq:uniintcondeq}
\end{align}
for $\psi_t(z)={1}_{\Delta}(t,z),\;\Delta\in\mathcal{B}_{[0,T]\times\Rr}$: 
$m(\Delta) = E[\Lambda^H(\Delta)] <\infty$, then $B^{\theta}$ and $M^\theta(\psi)$ are strongly orthogonal under $\Q$.
\end{theorem}

We recall that two $\GG$-martingales with values in $L^2(\Q)$ are strongly orthogonal if their product is a uniformly integrable $(\GG,\Q)$-martingale or, equivalently, if their quadratic variation process is a uniformly integrable $(\GG,\Q)$-martingale.

\bigskip
\begin{proof}
With $\epsilon\in[0,1]$, define
\[
X^{\epsilon}_t:=\epsilon B^{\theta}_t+M^\theta_t (\psi).
\]
We have
\begin{align*}
dX^{\epsilon}_t
&=-\alpha^{\epsilon}_tdt+\epsilon dB_t+\int_{\Rr}\psi_t(z)\tilde{H}(dt,dz),
\end{align*}
where
\[
\alpha^{\epsilon}_t=\epsilon \theta^B_t\lambda^B_t +\int_{\Rr}\psi_t(z)\theta^H_t(z)\nu(dz)\lambda^H_t.
\]
From Lemma 1.27 in \cite{OS07}, we know that if $Z_tX^{\epsilon}_t$ is a local $(\GG,\Prob)-$martingale, then $X^{\epsilon}_t$ is a local $(\GG,\Q)-$martingale. From Definition 1.28 and Example 1.29 in \cite{OS07}, and recalling that 
\[
dZ_t=Z_{t^-}\bigg{(}\theta^B_tdB_t+\int_{\Rr}\theta^H_t(z)\tilde{H}(dt,dz)\bigg{)},
\]
we get the following:
\begin{align*}
d(Z_tX^{\epsilon}_t)&=Z_{t^-}dX^{\epsilon}_t+X^{\epsilon}_{t^-}dZ_t+dZ_tdX^{\epsilon}_t\\
&=Z_{t^-}\bigg(\epsilon+X^{\epsilon}_{t^-}\theta^B_t\bigg)dB_t\\
&\quad+Z_{t^-}\int_{\Rr}\bigg(X^{\epsilon}_{t^-}\theta^H_t(z)+\psi_t(z)+\psi_t(z)\theta^H_t(z)\bigg)\tilde{H}(dt,dz).
\end{align*}
Thus, $X^{\epsilon}_t$ is a local $(\GG,\Q)-$martingale for all $\epsilon\in[0,1]$. In particular, $X^0_t=M^\theta_t(\psi)$ is a local $(\GG,\Q)-$martingale. Moreover, $B^{\theta}_t=X^1_t-M^\theta_t(\psi)$ is also a local $(\GG,\Q)-$martingale.

Since $B^{\theta}$ is a continuous local martingale ($B^\theta_0=0$), with quadratic variation $[B^{\theta},B^{\theta}]_t=[B,B]_t=\Lambda^B_t$ (the quadratic variation is invariant under equivalent measure change), then $B^{\theta}$ is a time-changed $(\GG,\Q)$-Brownian motion, see Theorem 16.4 in \cite{Kallenberg1997}. Hence, it is also a doubly stochastic Gaussian measure as in Definition \ref{definition:listA}.

As for $M^\theta (\psi)$, we can see that its quadratic variation is
\begin{align*}
[M^\theta(\psi),M^\theta(\psi)]_t=\int_0^t\int_{\Rr}\psi^2_s(z) H(ds,dz).\\
\end{align*}
Now, let $\psi_t(z)={1}_{(0,t]\times B}(t,z)$, for $t\in[0,T]$ and $B\in\mathcal{B}_{\Rr}$. Then,
\begin{align*}
\E_{\Q}\bigg[[M^\theta(\psi),M^\theta(\psi)]_T\bigg]&=\E_{\Q}\bigg[\int_0^T\int_{\Rr}{1}_{(0,t]\times B}(t,z) H(dt,dz)\bigg]\\
&=\E\bigg[Z_T\int_0^T\int_{\Rr}{1}_{(0,t]\times B}(t,z) H(dt,dz)\bigg].
\end{align*}
By H\"older's inequality, we have that
\begin{equation}
\E_{\Q}\bigg[[M^\theta(\psi),M^\theta(\psi)]_T\bigg]
\leq \bigg(\E[Z_T^2]\bigg)^{1/2}
\cdot\bigg(\E \bigg[H((0,t]\times B) ^2 \bigg]\bigg)^{1/2}
<\infty.\label{eq:notinfi}
\end{equation}
Hence $M^\theta(\psi)$ is a $(\GG,\Q)$- martingale, see e.g. Corollary to Theorem 27.II in \cite{Protter2005}.
Denote $\mathcal{B}$ a semi-ring generating $\mathcal{B}_{\Rr}$. We can regard the $\sigma-$algebra $\mathcal{B}_{(0,T]}$ as generated by the semi-ring of intervals of the form $(s,t]$, where $0\leq s< t\leq T.$ The $\sigma-$algebra $\mathcal{B}_{(0,T]\times\Rr}$ is generated by the semi-ring of sets $(s,t]\times A$, where $A\in\mathcal{B}.$
For an element $(s,t]\times A$ in the semi-ring, let $\psi={1}_{(0,t]\times A}$. Then we have
\begin{equation}
\tilde{H}^{\theta}((s,t]\times A) = M^\theta_t(\psi)-M^\theta_s(\psi).\label{eq:ring}
\end{equation}
By (\ref{eq:notinfi}) and (\ref{eq:ring}) $\tilde{H}^{\theta}$ is $\sigma-$finite on the semi-ring $\Prob-$a.s. (equivalently $\Q-$a.s.), hence we can uniquely extend (\ref{eq:ring}) to the $\sigma-$algebra $\mathcal{B}_{[0,T]\times\Rr}$, see Theorem 11.3 and Theorem 10.3 in \cite{PB95}. 
Hence, $\tilde{H}^{\theta}$ has the $(\GG,\Q)$-martingale property, conditionally orthogonal values with respect to $(\GG,\Q)$, and its variance measure is $\sigma-$finite. $\tilde{H}^{\theta}$ is clearly $\GG-$adapted by its definition, and $\tilde{H}^{\theta}$ is additive and $\sigma-$additive in $L^2(\Q)$ by its integral form and the condition on $\theta^H$. In conclusion, $\tilde{H}^{\theta}$ is a $(\GG,\Q)$-martingale random field with conditionally orthogonal values.

Finally, we show that $B^{\theta}$ and $M^\theta(\psi)$ are strongly orthogonal under $\Q$, for $\psi={1}_{\Delta}$ with $\Delta \in \mathcal{B}_{[0,T]\times \Rr}$: $m(\Delta)=\E[\Lambda^H(\Delta)]<\infty$. 
In fact, observe that  
\begin{align*}
B^{\theta}_t&:=B_t-\bigg{\langle} B,\int_0^{\cdot}\theta^B_sdB_s\bigg{\rangle}_t,\\
M^\theta_t(\psi) = \int_0^t\int_{\Rr} {1}_{\Delta}(s,z)  \tilde{H}^{\theta}(dt,dz)
&:=\int_0^t\int_{\Rr} {1}_{\Delta}(s,z)\tilde{H}(dt,dz)\\
&\quad-\bigg{\langle} \int_0^{\cdot}\int_{\Rr} {1}_{\Delta}(s,z) \tilde{H}(ds,dz),\int_0^{\cdot}\int_{\Rr}\theta^H_s(z)\tilde{H}(ds,dz)\bigg{\rangle}_t,
\end{align*}
Then 
\[
\bigg{\langle} B^{\theta}, M^\theta(\psi)\bigg{\rangle}_t
=\bigg{\langle}B, \int_0^{\cdot}\int_{\Rr} {1}_{\Delta}(s,z) \tilde{H}(dt,dz)\bigg{\rangle}_t=0,
\]
as a consequence of \ref{list:A5} in Definition \ref{definition:listA}.
From this we know that $[B^{\theta}, M^\theta(\psi)]_t$ is a $\Q$-local martingale and, by (\ref{eq:uniintcondeq}) and   Theorem 51.I in \cite{Protter2005}, we get that $[B^{\theta}, M^\theta(\psi)]$ is a uniformly integrable $\Q$-martingale. Then $B^{\theta}_t$ and $M^\theta_t(\psi)$ are strongly orthogonal square integrable martingales.
\end{proof}

\begin{remark}
If \eqref{eq:radoniko} is defined with $\theta\in \II_\FF$, then $Z^\theta$ is an $\FF$-adapted process. In this case, the fields $B^\theta$ and $\tilde H^\theta$ would be strongly orthogonal $(\FF,\Q)$-martingale random fields in the sense discussed above.
\end{remark}

Note that $\tilde{H}^{\theta}$ is not a doubly stochastic Poisson random field under $\Q$, in general.

\begin{corollary} 
 \label{cor:mycor}
Let $\tilde{H}^{\theta}$ and $Z$ be defined as in Theorem \ref{thm:girsanov}. If the stochastic field $\theta^H$ is deterministic, then $\tilde{H}^{\theta}$ is a $(\GG, \Q)-$centered doubly stochastic Poisson random field. Moreover, if $\theta^H_t(z)=\theta^H_t$, then the new jump measure and the new time distortion process are given by
\[
\nu^{\theta}(dz)=\nu(dz),\quad\quad \lambda^{\theta}_t(\omega)=\{1+\theta^H_t\}\lambda_t(\omega).
\]
If $\theta^H_t(z)=\theta^H(z)$, then the new jump measure and the new time distortion process are given by
\[
\nu^{\theta}(dz)=\{1+\theta^H(z)\}\nu(dz),\quad\quad \lambda^{\theta}_t(\omega)=\lambda_t(\omega).
\]
\end{corollary}

\begin{proof}
This can be shown by studying the characteristic function (under $\Q$) of $\tilde{H}^{\theta}(\Delta)$ for $\Delta\in\mathcal{B}_{[0,T]\times\Rr}$. 
\end{proof}


\section{Hedging under worst case scenario}

In this section we define explicitly the set of scenarios $\Qq_\MM$ considered in the definition of the risk measure in \eqref{risk-measure}:
\begin{equation*}
\label{rm}
\rho_t (\xi) := \esssup_{\Q \in \Qq_\MM} \: \E_\Q \big[ - \xi \vert \Mm_t \big], \quad t\in [0,T].
\end{equation*}
We consider the cases of information flow given by $\MM = \GG, \FF$. 
Considering the agent's perspective, it is natural to choose the filtration $\FF$ as model for the information flow. In fact $\GG$ carries the information of the whole process of stochastic time-change, which would be a form of anticipating information embedded in the information flow not reasonably available to an agent.
We study the optimisation problem \eqref{Y3} in both cases on the market \eqref{bond}-\eqref{stock} with the $\FF$- adapted coefficients $r, \alpha, \sigma, \gamma$.

\begin{definition}\label{def:admpro} 
Let the process $Z_t$, $t\in [0,T]$, be a $(\MM,\Prob)$-martingale defined by
\begin{align*}
&dZ_t
=Z_{t^-}\Big(\theta^B_tdB_t+\int_{\Rr}\theta^H_t(z)\tilde{H}(dt,dz)\Big),\\
&Z_0=1,
\end{align*}
for $\theta\in\mathcal{I}_{{\GG}}$ (with the notation $\theta_t(0) = \theta_t^B$ and $\theta_t(z) = \theta_t^H(z)$ for $z\in \R_0$) and $\theta^H_t(z)>-1$ $\Prob\times\Lambda-a.e.$
Consider the cases such that $Z_T\in L^2(\Prob)$ and (\ref{eq:uniintcondeq}) is satisfied. Moreover, for $K>0$,
\begin{align}
|\theta^B_t\lambda^B_t|<K,\quad 0\leq \theta^H_t(z)\sqrt{\lambda^H_t}< K\cdot z,\;z\in\Rr,\quad\quad\Prob\times dt-a.e.\label{eq:bound4}
\end{align}
Then the set of admissible scenarios is given by:
\begin{align*}
\mathcal{Q}_{\MM}:=\{\Q\sim \Prob \;|\;&\frac{d\Q}{d\Prob}=Z^{\theta}_T,\;
\theta \in \II_\MM\},
\end{align*}
where the Radon-Nikodym derivative are of the type above.
\end{definition}

We remark that $\Qq_\MM$ is a convex set. Moreover, $\Qq_\FF \subseteq \Qq_\GG$.

\bigskip
We recall that the hedging problem considered \eqref{Y}, with the chosen risk measure \eqref{risk-measure}, translates to the problem \eqref{Y3}:
\begin{align}
\label{Y3,2}
Y_t= \essinf_{\pi\in\Pi_\MM}\:  \esssup_{\Q\in\Qq_\MM} Y^{\pi, \Q}_t
\end{align}
with
\begin{align*}
Y^{\pi, \Q}_t :=& 
 {E}_{{\Q}}\bigg{[}
e^{-\int_t^Tr_sds}F-\int_t^Te^{-\int_t^sr_udu}\pi_s(\alpha_s-r_s)ds\nonumber\\
&\;-\int_t^Te^{-\int_t^sr_udu}\pi_s\sigma_sdB_s\nonumber\\
&\;-\int_t^T\int_{\Rr}e^{-\int_t^sr_udu}\pi_s\gamma_s(z)\tilde{H}(ds,dz)\bigg{|}\mathcal{M}_t\bigg{]} .
\end{align*}
Hence, a solution to \eqref{Y3} is given by $(\hat \pi, \hat\Q)\in \Pi_\MM \times \Qq_\MM$ such that
\begin{align}
\label{Ysolution,2}
Y_t = Y^{\hat\pi, \hat\Q}_t =  {E}_{\hat\Q}\bigg{[}
&e^{-\int_t^Tr_sds}F-\int_t^Te^{-\int_t^sr_udu}\hat\pi_s(\alpha_s-r_s)ds\nonumber\\
&-\int_t^Te^{-\int_t^sr_udu}\hat \pi_s\sigma_sdB_s\nonumber\\
&-\int_t^T\int_{\Rr}e^{-\int_t^sr_udu} \hat \pi_s\gamma_s(z)\tilde{H}(ds,dz)\bigg{|}\mathcal{M}_t\bigg{]}, \quad t\in [0,T].
\end{align}

\bigskip
Analogously, we define the set of admissible portfolio $\Pi_\MM$ with respect to the filtrations $\MM=\FF, \GG$. 
\begin{definition}
\label{def:adm} 
The portfolio $\pi:[0,T]\times\Omega\rightarrow{\R}$ is admissible if
\begin{enumerate}[(i)]
\item 
$\pi_t\gamma_t(z)>-1\quad$  $\Prob$-a.s.,
\item $\pi$ is $\MM$-predictable such that there exists a unique strong c\`adl\`ag $\MM$-adapted solution $V^{\pi}$ to the dynamics \eqref{value} on $[0,T]$,
\item
for all $\Q\in \Qq_\MM$ 
\begin{align*}
\E_{\Q}\bigg{[}\int_0^T \Big[|\alpha_t-r_t||\pi_t|+|\pi_t\sigma_t|^2 \lambda^B_t+\int_{\Rr}|\pi_t\gamma_t(z)|^2\nu(dz)\lambda^H_t \Big] dt\bigg{]}<\infty.
 \end{align*}
 \end{enumerate}
\end{definition}
Note that $\Pi_\MM$ is a convex set.

The solution to the problem \eqref{Y3,2} is studied via BSDEs and the comparison theorem. The two filtrations lead to different types of BSDEs.

\bigskip
\subsection{Flow of information $\GG$}

First of all we consider the filtration $\GG$ and the corresponding stochastic process:
\begin{align*}
&Y^{\pi,\theta}_t=\E_{\Q}\bigg{[}e^{-\int_t^Tr_sds}F-\int_t^Te^{-\int_t^sr_udu}\pi_s(\alpha_s-r_s)ds\nonumber\\
&-\int_t^Te^{-\int_t^sr_udu}\pi_s\sigma_sdB_s-\int_t^T\int_{\Rr}e^{-\int_t^sr_udu}\pi_s\gamma_s(z)\tilde{H}(ds,dz)\bigg{|}\mathcal{G}_t\bigg{]},
\label{eq:YQ}
\end{align*}
where $\pi\in\Pi_{\GG}$ and $\Q = \Q^\theta \in\mathcal{Q}_{\GG}$. 
By Theorem \ref{thm:girsanov}, define the $\GG$-martingale random fields: 
\[
dB^{\theta}_t:=dB_t-\theta^B_t\lambda^B_tdt,
\]
and
\[
\tilde{H}^{\theta}(dt,dz):=\tilde{H}(dt,dz)-\theta^H_t(z)\nu(dz)\lambda^H_t dt.
\]
So when $\MM=\GG,\,Y^{\pi,\Q}$ takes the form:
\begin{align}
Y^{\pi,\theta}_t
&=\E_{\Q}\bigg{[}e^{-\int_t^Tr_sds}F-\int_t^Te^{-\int_t^sr_udu}\pi_s\bigg{(}(\alpha_s-r_s)+\sigma_s\theta^B_s \lambda^B_s \nonumber\\
&\quad+\int_{\Rr}\gamma_s(z)\theta^H_s(z)\nu(dz)\lambda^H_s\bigg{)}ds-\int_t^Te^{-\int_t^sr_udu}\pi_s\sigma_sdB^{\theta}_s\nonumber\\
&\quad-\int_t^T\int_{\Rr}e^{-\int_t^sr_udu}\pi_s\gamma_s(z)\tilde{H}^{\theta}(ds,dz)\bigg{|}\mathcal{G}_t\bigg{]}\nonumber\\
&=\E_{\Q}\bigg{[}e^{-\int_t^Tr_sds}F-\int_t^Te^{-\int_t^sr_udu}\pi_s\bigg{(}(\alpha_s-r_s)+\sigma_s\theta^B_s\lambda^B_s \nonumber\\
&\quad+\int_{\Rr}\gamma_s(z)\theta^H_s(z)\nu(dz)\lambda^H_s\bigg{)}ds\bigg{|}\mathcal{G}_t\bigg{]}.\label{eq:exY}
\end{align}
Hence, 
\begin{align}
e^{-\int_0^tr_udu}Y^{\pi,\theta}_t 
= &\: \E_{\Q}\bigg{[}e^{-\int_0^Tr_sds}F-\int_0^Te^{-\int_0^sr_udu}\pi_s\nonumber\\
&\quad\cdot\bigg{(}(\alpha_s-r_s)+\sigma_s\theta^B_s\lambda^B_s+\int_{\Rr}\gamma_s(z)\theta^H_s(z)\nu(dz)\lambda^H_s\bigg{)}ds\bigg{|}\mathcal{G}_t\bigg{]} \label{eq:martin}\\
&\quad+\int_0^te^{-\int_0^sr_udu}\pi_s\bigg{(}(\alpha_s-r_s)+\sigma_s\theta^B_s\lambda^B_s+\int_{\Rr}\gamma_s(z)\theta^H_s(z)\nu(dz)\lambda^H_s\bigg{)}ds.
\nonumber
\end{align}
The martingale representation Theorem \ref{Teorem:G_martingales} applied to 
\begin{align}\label{eq:rename}
\xi&:=e^{-\int_0^Tr_sds}F-\int_0^Te^{-\int_0^sr_udu}\pi_s
\bigg{(}(\alpha_s - r_s)+\sigma_s\theta^B_s\lambda^B_s+\int_{\Rr} \gamma_s(z)\theta^H_s(z)\nu(dz)\lambda^H_s\bigg{)}ds
\end{align}
gives us the existence of the $\GG$-predictable integrands $Z^{\pi,\theta}$ and $U^{\pi,\theta}$ for the two corresponding stochastic integrals, so we have 
\begin{align}
e^{-\int_0^tr_udu}Y^{\pi,\theta}_t
&=\E_{\Q}[\xi|\mathcal{F}^{\Lambda}_T]+\int_0^tZ^{\pi,\theta}_sdB^{\theta}_s+\int_0^t\int_{\Rr}U^{\pi,\theta}_s(z)\tilde{H}^{\theta}(ds,dz)\nonumber\\
&+\int_0^te^{-\int_0^sr_udu}\pi_s\bigg{(}(\alpha_s-r_s)+\sigma_s\theta^B_s \lambda_s^B 
+\int_{\Rr}\gamma_s(z)\theta^H_s(z)\nu(dz)\lambda^H_s\bigg{)}ds\nonumber .
\end{align}
The It\^o formula allows to obtain the linear BSDE
\begin{align}
dY^{\pi,\theta}_t
=&\;\bigg{(}r_tY^{\pi,\theta}_t+\pi_t\bigg(\alpha_t-r_t\bigg)+\bigg(\pi_t\sigma_t-e^{\int_0^tr_udu}Z^{\pi,\theta}_t\bigg)\theta^B_t\lambda^B_t\nonumber\\
&+\int_{\Rr}\bigg(\pi_t\gamma_t(z)-e^{\int_0^tr_udu}U^{\pi,\theta}_t(z)\bigg)\theta^H_t(z)\nu(dz)\lambda^H_t\bigg{)}dt\nonumber\\
&+e^{\int_0^tr_udu}Z^{\pi,\theta}_tdB_t+\int_{\Rr}e^{\int_0^tr_udu}U^{\pi,\theta}_t(z)\tilde{H}(dt,dz),\label{eq:mainbsde}\\
Y^{\pi,\theta}_T=&\;F,\nonumber
\end{align}
the solution of which is guaranteed by Theorem \ref{linearGG} thanks to \eqref{eq:bound4}.
The generator of this BSDE is:
\begin{align}
g_\cdot(\lambda,y,z,u(\cdot),\pi,\theta)=&-yr-(\mu-r)\pi-(\pi\sigma-e^{\int_0^{\cdot}r_sds}z)\theta^B \lambda^B\nonumber\\
&-\int_{\Rr}(\pi\gamma_\cdot(x)-e^{\int_0^{\cdot}r_sds}u(x))\theta^H(x)\nu(dx)\lambda^H.\label{eq:driver}
\end{align}

The min-max type of problem corresponding to \eqref{Y3,2} arises in stochastic differential games. With the comparison Theorem \ref{thm:comparison} in hands, we can justify the proof of the following result due to \cite{OS11} in our setting for the driving noises considered in this paper. 

As short hand notation, denote $g(\pi_t,\theta_t)=g_\cdot(\lambda,y,z,u(\cdot),\pi_t,\theta_t)$. 
The solution of a BSDE with standard parameters $(\xi,g)$ is denote by  $({Y}, {Z},{U})$, for an optimal $\hat{\theta}$  the solution is denoted by $(Y^{\pi},Z^{\pi},U^{\pi})$, and for an optimal $\hat\pi$ the solution is then $(Y^{\theta},Z^{\theta},U^{\theta})$. 
The solution given in the case $g(\hat{\pi}_t,\hat{\theta}_t)$ is denoted $(\hat{Y},\hat{Z},\hat{U})$.

\begin{theorem}
\label{thm:main}
Let $(\xi,g)$ be standard parameters.
Suppose that for all $(\omega,t,\lambda,y,z,u)$ there exist $\hat{\pi}_t=\hat{\pi}(\omega,t,\lambda,y,z,u)$ and $\hat{\theta}_t=\hat{\theta}(\omega,t,\lambda,y,z,u)$ such that for all admissible portfolios $\pi\in \Pi_\GG$ and all admissible probability measures $\Q=\Q^\theta \in \Qq_\GG$, we have:
\begin{align}
g(\hat{\pi}_t,\theta_t)&\leq g(\pi_t,\theta_t)\leq g(\pi_t,\hat{\theta}_t),\label{eq:compdrivers}
\end{align}
for a.a. $(\omega,t)$. Assume that the conditions of Theorem \ref{thm:comparison} hold. 
Suppose $\hat{\pi}$ and $\hat{\theta}$ are admissible, and suppose that for all admissible $(\pi,\theta)$ there exists a unique solution to the BSDE with $(\xi,g(\pi_t,\theta_t))$ as terminal condition and generator, respectively. Then
\begin{align*}
\hat{Y}_t = Y^{\hat\pi}_t 
=\essinf_{\pi\in\Pi_\GG}Y^{\pi}_t=: Y_t=\esssup_{\Q\in\mathcal{Q}_\GG}\bigg{\{}\essinf_{\pi\in\Pi_\GG}Y^{\pi,\theta}_t\bigg{\}}=\esssup_{\Q \in\mathcal{Q}_\GG}Y^{\theta}_t.
\end{align*}
\end{theorem}
$ $\\

\begin{proof}
The proof is due to \cite{OS11}. By applying the comparison theorem to the solutions of the BSDEs of the couples of standard parameters $(F,g(\hat{\pi}_t,\theta_t))$,
$(F,g(\pi_t,\theta_t)),$ $(F,g(\pi_t,\hat{\theta}_t))$, by (\ref{eq:compdrivers}) we get that $Y^{\theta}_t\leq Y^{\pi,\theta}_t\leq Y^{\pi}_t$, 
for all admissible $(\pi,\theta)$, thus:
\begin{align}
\text{For all $\theta$:}&\quad Y^{\theta}_t\leq \essinf_{\pi\in\Pi}Y^{\pi,\theta}_t\quad \Prob\times dt-a.e.\label{eq:ineq1}\\
\text{For all $\pi$:}&\quad \esssup_{\theta\in\mathcal{Q}}Y^{\pi,\theta}_t\leq Y^{\pi}_t\quad \Prob\times dt-a.e.\label{eq:ineq2}
\end{align}
From definition of essential supremum and (\ref{eq:ineq1}), we get
\[
\hat{Y}_t 
\leq\esssup_{\Q\in\mathcal{Q}_\GG}Y^{\theta}_t
 = \esssup_{\Q\in\mathcal{Q}_\GG}\bigg{(}\essinf_{\pi\in\Pi_\GG}Y^{\pi,\theta}_t\bigg{)}.
\]
From (\ref{eq:ineq2}) and definition of essential infimum, we get
\[
Y_t=\essinf_{\pi\in\Pi_\GG}\bigg{\{}\esssup_{\Q\in\mathcal{Q}_\GG}Y^{\pi,\theta}_t\bigg{\}}\leq \essinf_{\pi\in\Pi_\GG}Y^{\pi}_t\leq \hat{Y}_t.
\]
Hence, we obtain the following chain of inequalities:
\begin{align*}
Y_t=\essinf_{\pi\in\Pi_\GG}\bigg{\{}\esssup_{\Q\in\mathcal{Q}_\GG }Y^{\pi,\theta}_t\bigg{\}}&\leq \essinf_{\pi\in\Pi_\GG}Y^{\pi}_t\leq \hat{Y}_t\\ 
&\leq\esssup_{\Q\in\mathcal{Q}_\GG}Y^{\theta}_t\leq\esssup_{\Q\in\mathcal{Q}_\GG}\bigg{(}\essinf_{\pi\in\Pi_\GG}Y^{\pi,\theta}_t\bigg{)}.
\end{align*}
Since $\sup(\inf)\leq\inf(\sup)$ we get equality between all terms. 
\end{proof}

\bigskip
We shall apply this result. The generator (\ref{eq:driver}) satisfies the conditions of $g^{(2)}$ in Theorem \ref{thm:comparison}. In fact for an admissible probability measure $\Q^\theta$ and an admissible $\pi$, we have:
\begin{align*}
g_t(\lambda&,y,\zeta,u(\cdot),\pi_t,\theta_t)\nonumber\\
=&-yr_t-\pi_t\bigg{[}\alpha_t-r_t+\sigma_t\theta^B_t \lambda^B +\int_{\Rr}\gamma_t(z)\theta^H_t(z)\nu(dz)\lambda^H\bigg{]}\\
&+e^{\int_0^tr_sds}\zeta \theta^B_t \sqrt{\lambda^B}\sqrt{\lambda^B}+\int_{\Rr}e^{\int_0^tr_sds}u(z)\sqrt{\lambda^H}\theta^H_t(z)\nu(dz)\sqrt{\lambda^H}.
\end{align*}
Moreover, condition \eqref{eq:compdrivers} leads to the study of the equations
\begin{equation*}
\frac{\partial g}{\partial \theta^B}(\hat{\pi},\hat{\theta})=0,\quad
\frac{\partial g}{\partial \theta^H}(\hat{\pi},\hat{\theta})=0,\quad
\frac{\partial g}{\partial \pi}(\hat{\pi},\hat{\theta})=0.
\end{equation*}
The determinant of the Hessian is null, and these equations correspond to a critical point.
Recall that $\Qq_\GG$ and $\Pi_\GG$ are convex.
These yield to the characterising equations for the optimal solution.

\bigskip
Summarising, we have the following result. 

\begin{theorem}
\label{final}
Let the reference filtration be $\GG$.
If $(\hat\pi , \hat\Q) \in \Pi_\GG \times \Qq_\GG$ satisfy the equations:
\begin{align}
&\Big(e^{\int_0^tr_sds}\hat{Z}_t-\hat{\pi}_t\sigma_t\Big) \lambda^B_t =0,\label{eq:hihi1}\\
&\int_{\Rr}\Big(e^{\int_0^tr_sds}\hat{U}_t(z)-\hat{\pi}_t\gamma_t(z)\Big)\nu(dz)\lambda^H_t =0,\label{eq:hihi2}\\
&(\alpha_t-r_t)+\sigma_t\hat\theta^B_t\lambda^B_t+\int_{\Rr}\gamma_t(z)\hat\theta^H_t(z)\nu(dz)\lambda^H_t =0,
\label{eq:hihi3}
\end{align}
where $(\hat{Z},\hat{U})\in\mathcal{I_{\GG}}$ are the integrands in the integral representation (Theorem \ref{theorem:ito_representation}):
\begin{align}
e^{-\int_0^Tr_tdt}F=\E_{\hat{\Q}}[e^{-\int_0^Tr_tdt}F|\mathcal{F}^{\Lambda}_T]+\int_0^T\hat{Z}_sdB^{\hat\theta}_s+\int_0^T\int_{\Rr}\hat{U}_s(z)\tilde{H}^{\hat\theta}(ds,dz), \label{eq:summary1}
\end{align}
then $(\hat\pi, \hat\Q)$ is the optimal solution of the problem \eqref{Ysolution}.
The optimal price process $\hat Y_t = Y^{\hat\pi, \hat\Q}_t = Y_t$, $t \in [0,T]$, is given by:
\begin{align}
\hat Y_t=\E_{\hat\Q} \Big[ e^{-\int_t^Tr_sds}F&-\int_t^Te^{\int_0^s r_udu}\hat{Z}_sdB^{\hat\theta}_s
-\int_t^T\int_{\Rr}e^{\int_0^s r_udu}\hat{U}_s(z)\tilde{H}^{\hat\theta}(ds,dz) \vert \Gg_t \Big] \label{opt-price-G}
\end{align}
where 
\[
dB^{\hat\theta}_t:=dB_t-{\hat\theta^B}_t\lambda^B_tdt
\]
is a $(\GG, \hat\Q)$-martingale and a time-changed $(\GG,\hat\Q)$-Brownian motion, and
\[
\tilde{H}^{\hat\theta}(dt,dz):=\tilde{H}(dt,dz)-{\hat\theta^H}_t(z)\nu(dz)\lambda^H_t dt.
\]
is a $(\GG,\hat\Q)$-martingale random field orthogonal to $B^{\hat\theta}$.
Under probability measure $\Prob$, the optimal price process follows the following dynamics:
\begin{equation}
\begin{split}
d \hat Y_t =  & \Big(  \hat Y _t  r_t + \hat  \pi_t \big( \alpha_t - r_t \big) \Big) dt + \hat \pi _t \sigma_t dB_t + \int_\Rr \hat \pi_t \gamma_t(z) \tilde H(dtdz)  \\  
 \hat Y_0 =& E_{\hat\Q} [e^{-\int_0^Tr_tdt}F|\mathcal{F}^{\Lambda}_T]. 
\end{split}
\label{opt-price-dyn-G}
\end{equation}
\end{theorem}
On the other side, if $(\hat\pi , \hat\Q) \in \Pi_\GG \times \Qq_\GG$ is an optimal solution, then the equations \eqref{eq:hihi1}-\eqref{eq:hihi3} are satisfied.

\begin{remark}
\label{G-comment1}
We observe that the optimal strategy $(\hat\pi, V^{\hat\pi}, C^{\hat\pi})$ is then given by the process $\hat\pi$ as characterised above, the value process $V^{\hat\pi}$ from \eqref{value}
has the initial value $v=V^{\hat\pi} _0 = E_{\hat\Q} [e^{-\int_0^Tr_tdt}F ]$, and the initial cost $C^{\hat\pi}_0= E_{\hat\Q} [e^{-\int_0^Tr_tdt}F|\mathcal{F}^{\Lambda}_T] - E_{\hat\Q} [e^{-\int_0^Tr_tdt}F]$ and is $C^{\hat\pi}_t =0$ for $t \in (0,T]$. 
\end{remark}

\begin{remark}
\label{G-comment2}
Observe that the optimal $\hat\Q$ is a martingale measure for the optimal price process $\hat Y$.
In fact from \eqref{opt-price-G}, we have that
\begin{align*}
e^{-\int_0^tr_sds}\hat{Y}_t=\E_{\hat{\Q}}[e^{-\int_0^Tr_sds}\hat{Y}_T|\mathcal{G}_t].\label{eq:YtGmart}
\end{align*}

\end{remark}

\bigskip
\bigskip

\subsection{Flow of information $\FF$} 
\label{sec:standard}

In the case of flow of information $\FF$, we see that the problem \eqref{Y3} leads to a different type of BSDE than the one considered so far.

Recall that 
$
\mu (dt, dz)= {1}_{\{0\}} B(dt,dz) +{1}_{\Rr}(z)\tilde{H}(dt,dz),
$ 
is an $(\FF,\Prob)$-martingale random field and that its conditional variance measure is:
\[
\langle \mu \rangle (dt,dz) =   \delta_{\{0\}}(dz) \lambda^B_s ds+  1_\Rr (z) \nu(dz)\lambda^H _sds.
\]

Moreover, by Theorem \ref{thm:girsanov} we define the $({\FF},\Q)-$martingale random field, $\mu^\theta$, by
\begin{align}
{\mu^\theta}(dt,dz)=\mu(dt,dz) - 
\beta_t(z)\langle \mu \rangle(dt,dz),
\label{eq:changemeasuregen}
\end{align}
where $\beta$ is given by:
$$
\beta_t(z) := \theta^B _t 1_{\{0\}}(z) +  \theta^H_t(z) 1_\Rr (z) .
$$

Note that $\langle \mu \rangle=\langle \mu^\theta \rangle$.

\bigskip
Let us consider the process $Y^{\pi, \Q}$ under filtration $\FF$ for $\pi\in \Pi_\FF$ and $\Q\in\Qq_\FF$:
\begin{align}
Y^{\pi,\theta}_t=\E_{\Q}\bigg{[}
&e^{-\int_t^Tr_sds}F-\int_t^Te^{-\int_t^sr_udu}\pi_s(\alpha_s-r_s)ds\nonumber\\
&-\int_t^Te^{-\int_t^sr_udu}\pi_s\sigma_sdB_s\nonumber\\
&-\int_t^T\int_{\Rr}e^{-\int_t^sr_udu}\pi_s\gamma_s(z)\tilde{H}(ds,dz)\bigg{|}\mathcal{F}_t\bigg{]}. 
\label{eq:expressbeforebsde}
\end{align}

Then, by direct computation and using Theorem \ref{thm:ito_representation2}, we obtain the following equalities:
\begin{equation*}
\begin{split}
e^{-\int_0^t r_s ds} Y^{\pi,\theta}_t
=&\E_{\Q}\bigg{[}
e^{-\int_0^T r_s ds }  F-\int_0^T\int_{{\R}} e^{-\int_0^s r_u du }  \pi_s\bigg[\bigg((\alpha_s-r_s)+\sigma_s\theta^B_s\bigg){1}_{\{0\}}(z)\\
&+\gamma_s(z)\theta^H_s(z){1}_{\Rr}(z)\bigg]\langle \mu^\theta \rangle(ds,dz)\bigg{|}\mathcal{F}_t\bigg{]}\\
&+\int_0^t\int_{{\R}}e^{-\int_0^s r_u du }    \pi_s\bigg[\bigg((\alpha_s-r_s)+\sigma_s\theta^B_s\bigg){1}_{\{0\}}(z)\\
&+\gamma_s(z)\theta^H_s(z){1}_{\Rr}(z)\bigg]\langle \mu^\theta \rangle(ds,dz)\\
=&\E_{\Q}\bigg{[}\xi_0+\int_0^T\int_{{\R}}Z_t{1}_{\{0\}}(z)+U_t(z){1}_{\Rr}(z) {\mu^\theta}(dt,dz)\bigg{|}\mathcal{F}_t\bigg{]}\\
&+\int_0^t\int_{{\R}}e^{-\int_0^s r_u du }   \pi_s\bigg[\bigg((\alpha_s-r_s)+\sigma_s\theta^B_s\bigg){1}_{\{0\}}(z)\\
&+\gamma_s(z)\theta^H_s(z){1}_{\Rr}(z)\bigg]\langle \mu^\theta\rangle(ds,dz).
\end{split}
\end{equation*}
Equivalently,
\begin{equation*}
\label{eq:gettingthere}
\begin{split}
 e^{-\int_0^t r_s ds } Y^{\pi,\theta}_t
=&\E_{\Q}[\xi_0|\mathcal{F}_t]+\int_0^t\int_{{\R}}Z_s{1}_{\{0\}}(z)+U_s(z){1}_{\Rr}(z) {\mu^\theta}(ds,dz)\\
&+\int_0^t\int_{{\R}}e^{-\int_0^s r_u du }  \pi_s\bigg[\bigg((\alpha_s-r_s)+\sigma_s\theta^B_s\bigg){1}_{\{0\}}(z)\\
&+\gamma_s(z)\theta^H_s(z){1}_{\Rr}(z)\bigg]\langle \mu^\theta\rangle(ds,dz).
\end{split}
\end{equation*}
Recall that the random variable $\xi_0 \in L^2(\Omega,\Ff, \Q)$ is orthogonal to the stochastic integrals. 
The equation above, together with \eqref{eq:changemeasuregen}, yields the BSDE:
\begin{align}
dY^{\pi,\theta}_t=&-\int_{{\R}}\bigg[\bigg{\{}-r_tY_t-\pi_t(\alpha_t-r_t)
+\theta^B_t\{e^{\int_0^tr_sds}Z_t-\pi_t\sigma_t\}\bigg{\}}{1}_{\{0\}}(z)\nonumber\\
&+\bigg{\{}e^{\int_0^tr_sds}U_t(z)-\pi_t\gamma_t(z)\bigg{\}}\theta^H_t(z){1}_{\Rr}(z)\bigg]\langle \mu \rangle(dt,dz)\nonumber\\
&+\int_{{\R}}\bigg[e^{\int_0^tr_sds}Z_t{1}_{\{0\}}(z)+e^{\int_0^tr_sds}U_t(z){1}_{\Rr}(z)\bigg] \mu (dt,dz)\label{eq:Ytwo}\\
&+e^{\int_0^tr_sds}d\E_{\Q}[\xi_0|\mathcal{F}_t]\nonumber\\
Y^{\pi,\theta}_T=&F.\nonumber
\end{align}

We remark that $\E_\Q [\xi_0\vert \Ff_t]$, $t\in [0,T]$, is an $(\FF,\Q)$-martingale orthogonal to $\mu^\theta$. Thus the process $\int_0^t e^{\int_0^s r_udu}d\E_{\Q}[\xi_0|\mathcal{F}_s]$, $t\in [0,T]$, is an $(\FF,\Q)$-martingale orthogonal to $\mu^\theta$.
By direct computation that this process is also an $(\FF, \Prob)$-martingale orthogonal to $\mu$.
In fact, for $A \in \Bb(\R)$, we have
\begin{align*}
\langle \int_0^\cdot e^{\int_0^s r_udu}d\E_{\Q}[\xi_0|\mathcal{F}_s] ,\mu(A)\rangle
&=\langle  \int_0^\cdot e^{\int_0^s r_udu}d\E_{\Q}[\xi_0|\mathcal{F}_s] , \mu^\theta (A) \rangle \\&+ \langle 
 \int_0^\cdot e^{\int_0^s r_udu}d\E_{\Q}[\xi_0|\mathcal{F}_s] , 
\int_0^{\cdot}\int_{\R}\beta_t(z)d\langle\mu\rangle(dt,dz)\rangle =0.
\end{align*}

Another way to look at \eqref{eq:expressbeforebsde} is by application of Theorem \ref{thm:girsanov}. In fact, setting $\psi_t:= e^{-\int_0^t r_s ds}$, $t\in [0,T]$, we have:
\begin{align}
\psi_tY^{\pi,\theta}_t 
= &\E_{\Q}\bigg{[}
\psi_TF-\int_t^T\psi_s\pi_s\bigg[(\alpha_s-r_s)ds-\sigma_sdB_s
-\int_{\Rr}\gamma_s(z)\tilde{H}(ds,dz)\bigg]\bigg{|}\mathcal{F}_t\bigg{]}\nonumber\\
=&\E_{\Q}\bigg{[}
\psi_TF-\int_t^T\psi_s\pi_s\bigg[(\alpha_s-r_s)ds+\sigma_s\bigg(dB^{\theta}_s+\theta^B_s\lambda^B_s ds\bigg)\nonumber\\
&+\int_{\Rr}\gamma_s(z)\bigg(\tilde{H}^{\theta}(ds,dz)+\theta^H_s(z)\nu(dz)\lambda^H_sds\bigg)\bigg]\bigg{|}\mathcal{F}_t\bigg{]}\nonumber
\end{align}
By use of the martingale property, we get
\begin{align}
\psi_tY^{\pi,\theta}_t
=&\E_{\Q}\bigg{[}
\psi_TF-\int_t^T\int_{{\R}}\psi_s\pi_s\bigg[\bigg((\alpha_s-r_s)+\sigma_s\theta^B_s\bigg){1}_{\{0\}}(z)\nonumber\\
&+\gamma_s(z)\theta^H_s(z){1}_{\Rr}(z)\bigg]\langle M\rangle(ds,dz)\bigg{|}\mathcal{F}_t\bigg{]}.\label{eq:prettynice2}
\end{align}
With $d\Q = Z_T d\Prob$ and $Z_T =  \mathcal{E}_T \big( \int_0^{\cdot} \int_\R \beta_s(z) \mu(ds,dz) \big)$, we recognize \eqref{eq:prettynice2} as the solution of
the linear BSDE of type, cf. \eqref{eq:expbsde}:
\begin{align*}
Y^{\pi,\theta}_t=\E\bigg{[}
\frac{\Psi_T}{\Psi_t}F-\int_t^T\int_{{\R}}\frac{\Psi_s}{\Psi_t}\pi_s\bigg[\bigg((\alpha_s&-r_s)+\sigma_s\theta^B_s\bigg){1}_{\{0\}}(z)\nonumber\\
+\gamma_s(z)\theta^H(&s,z){1}_{\Rr}(z)\bigg]\langle \mu\rangle(ds,dz)\bigg{|}\mathcal{F}_t\bigg{]}.
\end{align*}
Here we recall that $r$ is bounded, $\theta \in \II_\FF$ so that $\Q\in \Qq_\FF$, and \eqref{nec-value} holds.
By Lemma \ref{lem:carbonelemma} and Remark \ref{linearFF}, the $\Prob$-dynamics correspond to
\begin{align}
dY^{\pi,\theta}_t=&-\int_{{\R}}\bigg{\{}-r_t{1}_{\{0\}}(z)Y_t+\theta^B_t{1}_{\{0\}}(z)\bar{Z}_t
+\theta^H_t(z){1}_{\Rr}(z)\bar{U}_t(z)\nonumber\\
&-\pi_t\bigg[\bigg((\alpha_t-r_t)
+\sigma_t\theta^B_t\bigg){1}_{\{0\}}(z)
+\gamma_t(z)\theta^H_t(z){1}_{\Rr}(z)\bigg]\bigg{\}}\langle \mu \rangle(dt,dz)\label{eq:Yone}\\
&+\int_{{\R}}\bar{Z}_t{1}_{\{0\}}(z)+\bar{U}_t(z){1}_{\Rr}(z)\mu(dt,dz)+dN_t\nonumber\\
Y^{\pi,\theta}_T=&F,\nonumber
\end{align}
where $N$ is an $(\FF,\Prob)$-martingale orthogonal to $\mu$.
Comparing (\ref{eq:Yone}) and (\ref{eq:Ytwo}). We see that
\begin{align*}
\bar{Z}_t&=e^{\int_0^tr_sds}Z_t\\
\bar{U}_t(z)&=e^{\int_0^tr_sds}U_t(z),\\
N_t&=\int_0^t e^{\int_0^sr_udu}d\E_{\Q}[\xi_0|\mathcal{F}_s], \quad (N_0=0). 
\end{align*}



We can state the corresponding result to Theorem \ref{thm:main} for the case of information flow $\FF$.
Set $f(\pi_t, \theta_t) := f_t(\lambda,y,z,u,\pi_t, \theta_t)$ as short-hand notation.

\begin{theorem}
\label{thm:main2}
Let $(\xi,f)$ be standard parameters.
Suppose that for all $(\omega,t,\lambda, y,z,u)$ there exist $\hat{\pi}_t=\hat{\pi}(\omega,t, \lambda, y,z,u)$ and $\hat{\theta}_t=\hat{\theta}(\omega,t,\lambda,y,z,u)$ such that, for all admissible portfolios $\pi\in \Pi_\FF$ and all admissible probability measures $\Q=\Q^\theta \in \Qq_\FF$, we have:
\begin{align}
f(\hat{\pi}_t,\theta_t)&\leq f(\pi_t,\theta_t)\leq f(\pi_t,\hat{\theta}_t),\label{eq:compdriversF}
\end{align}
for a.a. $(\omega,t)$. Assume that the conditions of Theorem \ref{thm:carbonethm} hold, and that $\hat{\pi}$ and $\hat{\theta}$ are admissible. Suppose that for all admissible $(\pi,\theta)$ there exists a unique solution to the BSDE with $(\xi,f(\pi_t,\theta_t))$ as terminal condition and generator, respectively. Then
\begin{align*}
\hat{Y}_t=\essinf_{\pi\in\Pi_\FF}Y^{\pi}_t=: Y_t=\esssup_{\Q\in\mathcal{Q}_\FF}\bigg{\{}\essinf_{\pi\in\Pi_\FF}Y^{\pi,\theta}_t\bigg{\}}=\esssup_{\Q \in\mathcal{Q}_\FF}Y^{\theta}_t.
\end{align*}
\end{theorem}
\begin{proof}
The argument is the same as for Theorem \ref{thm:main}, but rely on the comparison theorem in the case $\FF$.
\end{proof}

From the BSDE in \eqref{eq:Yone}, we can see that
\begin{align}
f_t(y, \zeta, u(\cdot), \pi, \theta) = 
-&\Big{\{}-r_ty-\pi_t(\alpha_t-r_t)+\theta^B(\zeta-\pi_t\sigma_t)\Big{\}}{1}_{\{0\}}(z)\nonumber\\
&+\Big\{ u(z)-\pi_t\gamma_t(z)\Big{\}}\theta^H(z){1}_{\Rr}(z),
\label{eq:fgenerator}
\end{align}
where
$
\langle \mu \rangle (dt,dz) =   \delta_{\{0\}}(dz) \lambda^B(s) ds+  1_\Rr (z) \nu(dz)\lambda^H (s)ds.
$
So we obtain
\begin{align*}
g_\cdot(\lambda,y,\zeta,u(\cdot),\pi,\theta)
=&\Big{\{}-r y-\pi (\alpha - r)- \theta^B \pi \sigma \Big{\}} \lambda^B - \int_\Rr \pi \gamma(z) \theta^H(z)\nu(dz)\lambda^H\\
&+\theta^B\zeta {\lambda^B} + \int_\Rr \theta^H(z)u(z)\nu(dz) {\lambda^H}.
\end{align*}

We observe that BSDEs of the type \eqref{eq:Yone} with \eqref{eq:compdriversF} satisfy the conditions 
of Theorem \ref{thm:carbonethm}. 
In the same way as for Theorem \ref{final}, condition \eqref{eq:compdriversF} yields to the study of saddle points.
Hence we obtain the following result:

\begin{theorem}
\label{final2}
Let the reference filtration be $\FF$.
If $(\hat\pi , \hat\Q)\in \Pi_\FF \times \Qq_\FF$ satisfies the equations:
\begin{align}
&\Big( e^{\int_0^tr_sds}\hat{Z}_t-\hat{\pi}_t\sigma_t\Big) \lambda^B_t =0,\label{eq:hihi1F}\\
&\int_{\Rr}\Big(e^{\int_0^tr_sds}\hat{U}_t(z)-\hat{\pi}_t\gamma_t(z)\Big)\nu(dz)\lambda^H_t =0,\label{eq:hihi2F}\\
&(\alpha_t-r_t)+\sigma_t\hat\theta^B_t\lambda^B_t+\int_{\Rr}\gamma_t(z)\hat\theta^H_t(z)\nu(dz)\lambda^H_t =0,
\label{eq:hihi3F}
\end{align}
where $(\hat{Z},\hat{U})\in\mathcal{I_{\FF}}$ are the integrands in the integral representation:
\begin{align}
e^{-\int_0^Tr_tdt}F=\hat\xi_0+\int_0^T\hat{Z}_sdB^{\hat\theta}_s+\int_0^T\int_{\Rr}\hat{U}_s(z)\tilde{H}^{\hat\theta}(ds,dz),
\label{eq:pageref2}
\end{align}
and $\hat\xi_0 \in L^2(\Omega, \Ff, \hat\Q)$ is a random variable orthogonal to the stochastic integrals (cf. Theorem \ref{theorem:ito_representation2}), then $(\hat\pi, \hat\Q)$ is the solution of the optimisation problem \eqref{Y3,2}.
The optimal price process $\hat Y_t = Y^{\hat\pi, \hat\Q}_t = Y_t$, $t \in [0,T]$, is given by:
\begin{align}
\hat Y_t=\E_{\hat\Q} \Big[ e^{-\int_t^Tr_sds}F&-\int_t^Te^{\int_0^s r_udu}\hat{Z}_sdB^{\hat\theta}_s
-\int_t^T\int_{\Rr}e^{\int_0^s r_udu}\hat{U}_s(z)\tilde{H}^{\hat\theta}(ds,dz) \vert \Ff_t \Big],
\label{opt-price-F}
\end{align}
where 
\[
dB^{\hat\theta}_t:=dB_t-{\hat\theta^B}_t\lambda^B_tdt,
\]
and
\[
\tilde{H}^{\hat\theta}(dt,dz):=\tilde{H}(dt,dz)-{\hat\theta^H}_t(z)\nu(dz)\lambda^H_t dt.
\]
These are orthogonal $(\FF,\hat\Q)$-martingale random fields.
Under probability measure $\Prob$, the optimal price process has the following dynamics:
\begin{equation}
\begin{split}
d \hat Y_t =  & \Big(  \hat Y _t  r_t + \hat  \pi_t \big( \alpha_t - r_t \big) \Big) dt + \hat \pi _t \sigma_t dB_t 
 + \int_\Rr \hat \pi_t \gamma_t(z) \tilde H(dt,dz) + e^{\int_0^t r_s ds} d \E_{\hat\Q} [\hat\xi_0 \vert \Ff_t] \\  
 \hat Y_0 =& E_{\hat\Q} [e^{-\int_0^Tr_tdt}F ]. 
\end{split}
\label{opt-price-dyn-F}
\end{equation}
\end{theorem}

\begin{remark}
\label{F-comment1}
We observe that the optimal strategy $(\hat\pi, V^{\hat\pi}, C^{\hat\pi})$ is then given by the process $\hat\pi$ as characterised above. The wealth $V^{\hat\pi}$ 
on the market has the initial value $V^{\hat\pi} _0 = v= E_{\hat\Q} [e^{-\int_0^Tr_tdt}F ]$, and the cost is a process $C^{\hat\pi}_t$
, $t \in [0,T]$, given as follows.
Observe that 
$$
d\hat Y _t = dV^{\hat \pi}_t + \big( \hat Y_t - V^{\hat\pi}_t \big)r_t dt + e^{\int_0^t r_s ds} d \E_{\hat \Q} [\hat \xi_0 \vert \Ff_t],
$$
then 
$$
dC^{\hat \pi}_t = C^{\hat \pi}_tr_tdt + e^{\int_0^t r_s ds} d \E_{\hat \Q} [\hat \xi_0 \vert \Ff_t]; 
\quad C^{\hat \pi}_0 = 0.
 $$
which results in $C^{\hat \pi}_t = e^{\int_0^t r_s ds}  \big( \E_{\hat \Q} [\hat \xi_0 \vert \Ff_t] -  \E_{\hat \Q} [\hat \xi_0 ] \big).
$
\end{remark}

\begin{remark}
\label{F-comment2}
Observe that the optimal $\hat\Q$ is a martingale measure for the optimal price process $\hat Y$.
In fact, the price process is given by:
\begin{align*}
e^{-\int_0^tr_sds}\hat{Y}_t=\E_{\hat{\Q}}[e^{-\int_0^Tr_sds}\hat{Y}_T|\mathcal{F}_t].\label{eq:YtFmart}
\end{align*}
\end{remark}

\begin{remark}
\label{F-comment3}
The characterizing equations \eqref{eq:hihi1}-\eqref{eq:hihi3} and  \eqref{eq:hihi1F}-\eqref{eq:hihi3F} are formally the same, being the difference on the measurability properties of the processes involved. 
Denote $\hat \Q_\GG = \Q^{\hat\pi} \in \Qq_\GG$ in Theorem \ref{final} and $\hat \Q_\FF = \Q^{\hat\pi} \in \Qq_\FF$ in Theorem \ref{final2}.
From equations \eqref{eq:hihi3} and \eqref{eq:hihi3F} combined with Theorem \ref{thm:girsanov}, we can see that
$\hat\Q_{\GG \vert \FF} = \hat \Q_\FF $.
\end{remark}


\section{Conclusions and example}
\label{sec:analen}

With the intent of finding a hedging strategy in the incomplete market \eqref{bond}-\eqref{stock} we have studied the optimization problem \eqref{Y}. We have developed the solution using BSDEs and their comparison theorems. 
for this we refer to \cite{DS2014} and \cite{Carbone2008}, of which we adapt the results.
For our approach it is crucial the result of \cite{OS11} developed for stochastic differential games. Indeed we have transformed the hedging problem under model uncertainty in a min-max type problem \eqref{Y3} by exploiting the explicit representation of the risk-measure considered \eqref{risk-measure}.

\bigskip
The noises considered are naturally linked to two different filtrations. The filtration $\GG$ captures all the statistical properties of the noises, allowing to exploit the underlying Gaussian and Poisson structure, see Definition \ref{definition:listA}. The filtration $\FF$ is substantially the filtration generated by the noises. We study the problem \eqref{Y3} with respect to both situations and we observe that $\Ff_T = \Gg_T$. 
Correspondingly, we have proposed two BSDEs related to the the two filtered probability spaces. The terminal condition is the same. The results obtained show differences in terms of adaptability of the solutions and the structure of the solution itself, cf. \eqref{opt-price-dyn-G} and \eqref{opt-price-dyn-F}.

Both set-ups lead to hedging strategies $(\hat\pi, V^{\hat\pi}, C^{\hat\pi})$ with presence of cost process. In the case of $\GG$, the cost process accounts for the anticipated knowledge of the time-change. 
In the case of $\FF$, the cost process represents the spread  between the perfect hedge and the best self-financing strategy. Here the distances are evaluated in terms of the risk-measure \eqref{risk-measure}.

From the methodological point of view, we remark that the BSDEs presented for the study in the case of filtration $\FF$ are based on the properties of the martingale random fields and we recall that our noises are martingale random fields with respect to both filtrations. 
In fact in this framework of general martingales we can see the correspondence between the two set-ups generated by the two filtrations.

As explained earlier, from a financial modeling perspective it is better suited to consider the information flow given by $\FF$.

\begin{example} {\bf Toy example.}
Let $e^{-\int_0^Tr_tdt}F$ is $\mathcal{F}^{\Lambda}_T-$measurable ($r$ deterministic). 

In the information flow $\GG$, the integral representation \eqref{eq:summary1} shows that the optimal integrands $(\hat Z, \hat U)\in \II_\GG$ are null and that the optimal strategy presents $\hat\pi=0$, 
and $v= \E_{\hat\Q_\GG} [e^{-\int_0^Tr_tdt}F]$.
Here the values $(\hat\theta^B, \hat\theta^H)$ for the probability measure $\hat\Q_\GG = \Q^{\hat\theta} \in \Qq_\GG$ are given by \eqref{eq:hihi3}. See Theorem \ref{final}.
The cost process is $C^{\hat\pi} _0 = e^{-\int_0^Tr_tdt}F - \E_{\hat\Q_\GG} [e^{-\int_0^Tr_tdt}F]$, $C^{\hat\pi} _t = 0$, $t\in (0,T]$.

Consider the case of information flow $\FF$.
Denote $\HH_\FF \subset L^2(\Omega,\Ff, \hat\Q_\FF)$ and $\HH_\GG \subset L^2(\Omega,\Ff, \hat\Q_\GG)$ the spaces generated by the integrals $\int_0^T \int_\R \phi_t(z) \mu(dt,dz)$ for all $\phi \in \II_\FF$ and $\phi \in \II_\GG$, respectively.
Being $\hat\Q_\FF = \hat\Q_{\GG\vert\FF}$ (see Remark \ref{F-comment3}), then $\HH_\FF \subset \HH_\GG$.
Hence, 
$$
L^2  (\Omega,\Ff, \hat\Q_\GG) \ominus \HH_\FF \supset L^2  (\Omega,\Ff, \hat\Q_\GG) \ominus \HH_\GG \ni e^{-\int_0^T r_t dt} F.
$$
Hence $\hat \xi_0 = e^{-\int_0^T r_t dt} F$ and $(\hat Z, \hat U)\equiv 0$ in the representation \eqref{eq:pageref2}.
The optimal strategy is then $\hat\pi=0$, $v= \E_{\hat\Q_\FF} [e^{-\int_0^Tr_tdt}F]$, and 
$C^{\hat\pi} _t = \E_{\hat\Q_\FF} \big[  e^{-\int_0^Tr_tdt}F \vert \Ff_t] - \E_{\hat\Q_\FF} [e^{-\int_0^Tr_tdt}F]$, $t\in [0,T]$.

\end{example}

\bigskip

In line with Remarks \ref{G-comment2} and \ref{F-comment2}, we see that the optimal measures $\Q_\FF$ and  $\Q_\GG$ are risk-neutral in the given market. In fact, applying \eqref{eq:hihi3} or \eqref{eq:hihi3F}, i.e.
\[
(\alpha_t-r_t)+\sigma_t\hat\theta^B_t+\int_{{\R}_0}\gamma_t(z)\hat\theta^H_t(z)\nu(dz)\lambda_t=0,
\]
we have
\begin{align*}
d\left(e^{-\int_0^tr_sds}S^{(1)}_t\right)
&=e^{-\int_0^tr_sds}S^{(1)}_t\bigg[(\alpha_t-r_t)dt+\sigma_tdB_t
+\int_{{\R}_0}\gamma_t(z)\tilde{H}(dt,dz)\bigg]\\
&=e^{-\int_0^tr_sds}S^{(1)}_t\bigg[\sigma_tdB^{\hat\theta}_t+\int_{{\R}_0}\gamma_t(z)\tilde{H}^{\hat\theta}(dt,dz)\bigg].
\end{align*}
This result is consistent with the observations of \cite{OS11-2} in the context of dynamics given by a jump diffusion and in the literature related to risk-minimizing strategies.

\bigskip
\noindent
{\bf Acknowlegdements}\\
We acknowledge the support of the  Centre of Advanced Study (CAS) at the Norwegian Academy of Science and Letters that has hosted and funded the research project Stochastics in Environmental and Financial Economics (SEFE) during the academic year 2014/15.


\bibliography{DK-Aarhus200415.bbl}

\end{document}